\documentclass[11pt]{amsart}
\usepackage{tikz-cd} 
\usepackage{amsmath}
\usepackage{amsthm}
\usepackage{amsfonts}
\usepackage[margin=1.2in]{geometry}

\newtheorem{thm}{Theorem}
\newtheorem{defn}[thm]{Definition}
\newtheorem{lem}[thm]{Lemma}
\newtheorem{prop}[thm]{Proposition}
\newtheorem{cor}[thm]{Corollary}
\newtheorem{exmp}[thm]{Example}
\newtheorem{rem}[thm]{Remark}
\newtheorem{convention}{Convention}
\definecolor{cadmiumgreen}{rgb}{0.0, 0.42, 0.24}
\definecolor{darkpastelgreen}{rgb}{0.01, 0.75, 0.24}

\newcommand{\map}[3]{{#1}:{#2}\longrightarrow{#3}}
\renewcommand{\a}{\alpha}
\renewcommand{\b}{\beta}

\newcommand{\len}[1]{\ell(#1)}
\newcommand{\N}{\mathbb{N}}
\newcommand{\pt}{\partial}

\newcommand{\Il}{\mathsf{L}}
\newcommand{\Ir}{\mathsf{R}}
\newcommand{\pow}[2]{{#1}^{#2}}

\newcommand{\emptypartition}{\varepsilon}
\newcommand{\oblak}{\mathrm{Obk}}
\newcommand{\obi}[2]{#1^{#2}}

\newcommand{\Pset}{\mathcal{P}}
\newcommand{\D}{\mathfrak{D}}
\newcommand{\df}{\partial}
\newcommand{\U}{\mathcal{U}}
\newcommand{\Nil}{\mathcal{N}}
\newcommand{\Fset}{\mathcal{F}}
\newcommand{\AP}{\mathcal{A}}
\newcommand{\BP}{\mathcal{B}}
\newcommand{\two}[1]{\mu_2(#1)}
\newcommand{\Wset}{\mathcal{W}}
\newcommand{\w}{\omega}
\newcommand{\supp}{\mathcal{S}}
\newcommand{\Ba}{\mathcal{B}}
\newcommand{\Cm}{\mathcal{C}}
\newcommand{\Uc}{\mathcal{U}}
\newcommand{\Gc}{\mathcal{G}}
\DeclareMathOperator{\GL}{GL}
\DeclareMathOperator{\im}{Im}
\DeclareMathOperator{\rank}{rank}

\newcommand{\RRset}{\mathcal{Q}}
\newcommand{\emptyf}{\emptypartition}
\newcommand{\burge}[1]{\Omega(#1)}
\newcommand{\pti}[2]{\pt^{#2}{#1}}
\newcommand{\Des}[1]{\mathrm{Des}(#1)}
\newcommand{\des}[1]{\mathrm{des}(#1)}
\newcommand{\maj}[1]{\mathrm{maj}(#1)}

\newcommand{\rev}[1]{\mathrm{val}_{#1}}
\newcommand{\lev}[1]{\mathrm{val}_{#1}}
\newcommand{\ran}[1]{\mathrm{ann}_{#1}}
\newcommand{\lan}[1]{\mathrm{ann}_{#1}}

\usetikzlibrary{tikzmark}
\newcounter{pairctr}
\newcommand{\Lpair}[1]{\stepcounter{pairctr}\tikzset{tikzmark prefix=\thepairctr}\tikzmark{start}#1\tikzmark{stop}\tikz[remember picture, overlay]{
\draw[->]([shift={(.5ex,2.5ex)}]pic cs:start) to[bend left=20] ([shift={(-.5ex,2.5ex)}]pic cs:stop);}}
\newcommand{\Rpair}[1]{\stepcounter{pairctr}\tikzset{tikzmark prefix=\thepairctr}\tikzmark{start}#1\tikzmark{stop}\tikz[remember picture, overlay]{
\draw[<-]([shift={(.5ex,2.5ex)}]pic cs:start) to[bend left=20] ([shift={(-.5ex,2.5ex)}]pic cs:stop);}}

\newcommand\ara[1]{\arrow{r}{#1}}
\newcommand\arb[1]{\arrow[swap]{r}{#1}}

\newcommand\adp{\arrow[swap]{d}{\pt}}

\usepackage{amssymb}
\usepackage{graphicx}

\title{A Proof of the Box Conjecture for Commuting Pairs of Matrices}
\author[J.~Irving]{John Irving}
\address{Department of Mathematics and Computing Science, Saint Mary's University, 923 Robie St, Halifax, Nova Scotia, Canada B3N 1Z9}
\email{john.irving@smu.ca}

\author[T.~Ko\v{s}ir]{Toma\v z Ko\v sir}
\address{Faculty of Mathematics and Physics, University of Ljubljana, Jadranska 19, SI-1000 Ljub\-lja\-na, Slovenia}
\email{tomaz.kosir@fmf.uni-lj.si}

\author[M.~Mastnak]{Mitja Mastnak}
\address{Department of Mathematics and Computing Science, Saint Mary's University, 923 Robie St, Halifax, Nova Scotia, Canada B3N 1Z9}
\email{mitja.mastnak@smu.ca}

\date{\today}

\subjclass{15A27, 05A17, 15A21, 13E10, 14A25}

\begin{document}

\begin{abstract}
    We prove the Box Conjecture for pairs of commuting nilpotent matrices, as formulated by Iarrobino et al \cite{IKvSZ}. This describes the Jordan type of the dense orbit in the nilpotent commutator of a given nilpotent matrix. Our main tool is the Burge correspondence between the set of all partitions and a set of binary words \cite{Bur-1, Bur-2}.  For connection with the algebraic and geometric setup of matrices and orbits we employ some of Shayman's results on invariant subspaces of a nilpotent matrix \cite{Sha-1,Sha-2}. Our proof is valid over an arbitrary field.
\end{abstract}

\maketitle

\section{Introduction}

The study of commuting matrices has a long history (see e.g. \cite{GP, Ger, Gur, SuTy}). As such it is somewhat surprising that the question of which pairs of conjugacy classes can occur for pairs of commuting matrices has been considered only relatively recently. The question reduces to the case of commuting nilpotent matrices. Fix an infinite ground field $F$.  For a given nilpotent matrix $B$ (over $F$), the set $\Nil_B$ of all nilpotent matrices commuting with $B$ is an irreducible algebraic variety \cite[Lem. 2.3]{Bas-2}. Its intersection with one of the conjugacy classes of nilpotent matrices is a dense open subset. Let $\Pset$ be the set of all partitions of natural numbers together with the empty partition $\emptypartition$. The partition $P\in\Pset$  determined by the Jordan canonical form of $B$ is called the \emph{Jordan type} of $B$. Following Panyushev \cite{Pan} we denote by $\D(P)$ the partition determining the Jordan type of the dense conjugacy class in $\Nil_B$.

The map $\D:\Pset\to\Pset$ has been studied by several authors; see e.g.~\cite{Bas-I,BI,BIK,Kha-1,Kha-2,IKvSZ, KO,Pan} and also the recent survey by Khatami \cite{Kha-S}. Map $\D$ is idempotent \cite[Thm. 6]{KO} and its image is the set $\RRset$ of \emph{super-distinct} partitions, i.e., those whose parts differ by at least two \cite[Thm. 1.12]{BI},\cite[Thm. 2.1]{Pan}. The number of parts of $\D(P)$ was determined by Basili~\cite[Prop. 2.4]{Bas-1}. Oblak later identified the largest part of $\D(P)$ \cite[Thm. 13]{Obl-1}  and Khatami  the smallest \cite[Thm. 4.1]{Kha-2}. Oblak also conjectured a recursive process to compute $\D(P)$. Iarrobino and Khatami \cite{IK} showed that Oblak's process produces a lower bound (in dominance order), and Basili \cite{Bas-I} proved that the algorithm in fact yields $\D(P)$.

The inverse image under $\D$ is  not as well understood, owing to the apparent difficulty of unwinding Oblak's process. Iarrobino, Khatami, Van Steirteghem, and Zhao \cite{IKvSZ} described $\D^{-1}(Q)$  for partitions $Q\in\RRset$ having at most two parts and they also  formulated the \emph{Box Conjecture} \cite[Conj. 4.11]{IKvSZ} regarding the structure of $\D^{-1}(Q)$ in general.   The name of the conjecture comes from its precise statement: Given a partition $Q=(q_1,q_2,\ldots,q_k)\in\RRset$ with parts $q_1>q_2>\cdots>q_k$ differing by at least 2,   the  elements of  $\D^{-1}(Q)$ can be arranged in a box (i.e. array) of size $q_k\times (q_{k-1}-q_k-1)\times \cdots\times (q_1-q_2-1)$ such that the partition in the $(i_1,i_2,\cdots,i_k)$-th position has exactly $\sum_{j=1}^k i_j$ parts.

We prove the Box Conjecture and provide further insight into the structure of partitions in $\D^{-1}(Q)$. Our principal algebraic tool is the structure theory of invariant subspaces as presented by Shayman \cite{Sha-1,Sha-2}, but our argument hinges on a combinatorial correspondence due to Burge~\cite{Bur-1,Bur-2} that encodes  partitions as binary words.

%\stkout{We assume that the underlying field $F$ is an infinite field of characteristic $0$ or $p>n$, where $n$ is the dimension of the underlying vector space. We use in our proofs the facts that map $\D$ is idempotent and that its image is equal to $\RRset$. The fact that these results hold under our assumptions on the underlying field is explained in} \cite[\stkout{Rmk. 2.7}]{IKvSZ} \stkout{ (see also} \cite{BIK,IK}\stkout{)}. 
%\stkout{here}. 
Let $\map{\Omega}{\Pset}{\Wset}$ denote the Burge encoding, where $\Wset=(\a^*\b)^*\a$ is the set of words on the alphabet $\{\a,\b\}$ that end with a single $\a$.  (The encoding will be described in detail in Section~\ref{sec:groundwork}.)  Our main result is the following:

\begin{thm}
\label{thm:main}
Suppose $P \in \Pset$ has Burge code $\burge{P}=\w_1\cdots\w_n$.   Then $\D(P) = (q_1, q_2,\ldots,q_k)$, where $q_1 > q_2 > \cdots > q_k$ is the complete list of  indices $q$ for which $\w_q=\b$ and $\w_{q+1}=\a$.  
\end{thm}

For instance, we will find that $P=(5,3,2,2,1)$ has Burge code $\burge{P} =\b\a\b\a\a\a\b\b\b\a$  and thus Theorem~\ref{thm:main} gives $\D(P)=(9,3,1)$.  
In general, $\burge{P}$ can be obtained from $P$ via a straightforward iterative process, so Theorem~\ref{thm:main}  provides another recursive algorithm for computing $\D(P)$, distinct from (but related to) Oblak's.

The $q_i$ in Theorem~\ref{thm:main} clearly differ pairwise by at least 2,  so the theorem immediately gives $\D(P) \in \RRset$ for all $P$. 
Moreover, a simple characterization of super-distinct partitions in terms of the Burge correspondence (Proposition~\ref{prop:characterization}) will reveal that  Theorem~\ref{thm:main} implies $P \in \RRset$ if and only if $\D(P)=P$.    
  
We will also see that the number of copies of $\b$ in $\burge{P}$ equals the number of parts of $P$ (Proposition~\ref{prop:stats}).  The Box Conjecture then follows  from Theorem~\ref{thm:main} simply by specifying the words in $\Wset$ that contain substring $\b\a$  at prescribed positions.

\begin{cor}
\label{cor:boxconjecture}
Let $Q=(q_1,q_2,\ldots,q_k) \in \RRset$ and set $\delta_1 = q_k$ and $\delta_i = q_{k-i+1}-q_{k-i+2}-1$ for $2 \leq i \leq k$.  Then $\D^{-1}(Q)$ is of size $\delta_1\delta_2 \cdots\delta_k$ and consists of precisely those partitions whose Burge code is of the form
\begin{equation}
\label{eq:code}
\a^{\delta_1-i_1}\b^{i_1}\a^{\delta_2-i_2+1}\b^{i_2}
\a^{\delta_3-i_3+1} \b^{i_3} \cdots \a^{\delta_k-i_k+1} \b^{i_k}\a,
\end{equation}
where $(i_1,\ldots,i_k) \in [1,\delta_1] \times [1,\delta_2] \times \cdots \times [1,\delta_k]$.  Moreover, the partition determined by~\eqref{eq:code} has exactly $\sum_j i_j$ parts.
%where $1 \leq i_1 \leq \delta_1$ and $1 \leq i_j < \delta_j$ for $j=2,\ldots,r$$ 
\end{cor}

We noted above that super-distinct partitions are stable under $\D$, that is, $Q \in \D^{-1}(Q)$ for all $Q \in \RRset$.  This  corresponds to taking $(i_1,\ldots,i_k)=(1,\ldots,1)$ in Corollary~\ref{cor:boxconjecture}.

\begin{exmp}\label{exmp-3}
For $Q=(10,7,3) \in \RRset$ we have $(\delta_1,\delta_2,\delta_3)=(3,7-3-1,10-7-1)=(3,3,2)$, so $\D^{-1}(Q)$ contains $18$ partitions indexed by coordinates in the box $[1,3] \times [1,3] \times [1,2]$. These are displayed in Figure~\ref{fig:boxexample}. \end{exmp}

\begin{figure}
\label{fig:boxexample}
\centering
{
\begin{tabular}{l|l|l|l}
$(i_1,i_2,i_3)$ &  code $\w$ &  partition $\Omega^{-1}(\w)$ & \# parts\\
\hline
  $ (1, 1, 1) $ & $ \a\a\b\a\a\a\b\a\a\b\a $ & 
  $ [ 10,7,3] $ & $3 $ \\  
  $ (2, 1, 1) $ & $ \a\b\b\a\a\a\b\a\a\b\a $ &
  $ [ 10,7,2,1] $ & $4 $ \\  
  $ (3, 1, 1) $ & $ \b\b\b\a\a\a\b\a\a\b\a $ & 
  $ [ 10,7,\pow{1}{3}] $ & $5 $ \\  
  $ (1, 2, 1) $ & $ \a\a\b\a\a\b\b\a\a\b\a $ & 
  $ [ 10,5,3,2] $ & $4 $ \\  
  $ (2, 2, 1) $ & $ \a\b\b\a\a\b\b\a\a\b\a $ & 
  $ [ 10,4,3,2,1] $ & $5 $ \\  
  $ (3, 2, 1) $ & $ \b\b\b\a\a\b\b\a\a\b\a $ & 
  $ [ 10,4,3,\pow{1}{3}] $ & $6 $ \\  
  $ (1, 3, 1) $ & $ \a\a\b\a\b\b\b\a\a\b\a $ & 
  $ [ 10,5,\pow{2}{2},1] $ & $5 $ \\  
  $ (2, 3, 1) $ & $ \a\b\b\a\b\b\b\a\a\b\a $ & 
  $ [ 10,5,2,\pow{1}{3}] $ & $6 $ \\  
  $ (3, 3, 1) $ & $ \b\b\b\a\b\b\b\a\a\b\a $ & 
  $ [ 10,5,\pow{1}{5}] $ & $7 $ \\  
  $ (1, 1, 2) $ & $ \a\a\b\a\a\a\b\a\b\b\a $ &
  $ [ 9,5,\pow{3}{2}] $ & $4 $ \\  
  $ (2, 1, 2) $ & $ \a\b\b\a\a\a\b\a\b\b\a $ & 
  $ [ 9,\pow{4}{2},2,1] $ & $5 $ \\  
  $ (3, 1, 2) $ & $ \b\b\b\a\a\a\b\a\b\b\a $ & 
  $ [ 9,\pow{4}{2},\pow{1}{3}] $ & $6 $ \\  
  $ (1, 2, 2) $ & $ \a\a\b\a\a\b\b\a\b\b\a $ & 
  $ [ 9,5,\pow{2}{3}] $ & $5 $ \\  
  $ (2, 2, 2) $ & $ \a\b\b\a\a\b\b\a\b\b\a $ & 
  $ [ 9,4,3,2,\pow{1}{2}] $ & $6 $ \\  
  $ (3, 2, 2) $ & $ \b\b\b\a\a\b\b\a\b\b\a $ & 
  $ [ 9,4,3,\pow{1}{4}] $ & $7 $ \\  
  $ (1, 3, 2) $ & $ \a\a\b\a\b\b\b\a\b\b\a $ & 
  $ [ 9,5,\pow{2}{2},\pow{1}{2}] $ & $6 $ \\  
  $ (2, 3, 2) $ & $ \a\b\b\a\b\b\b\a\b\b\a $ & 
  $ [ 9,5,2,\pow{1}{4}] $ & $7 $ \\  
  $ (3, 3, 2) $ & $ \b\b\b\a\b\b\b\a\b\b\a $ & 
  $ [ 9,5,1^6] $ & $8 $ 
 \end{tabular}
 }
 \caption{The elements of $\D^{-1}(10,7,3)$ and the corresponding Burge codes, indexed as in Corollary~\ref{cor:boxconjecture} by their coordinates $(i_1,i_2,i_3) \in [1,3] \times [1,3] \times [1,2]$.}
 \end{figure}

Theorem~\ref{thm:main} will be proved in Section~\ref{sec:main} after we have established the necessary combinatorial framework in Section~\ref{sec:groundwork}.  Our proof  does not depend on the nature of the underlying field $F$, except insofar as the definition of $\D(P)$ requires the assumption that $F$ is infinite. (We explain an extension to  finite fields in  \S\ref{finite fields}.) In Section~\ref{sec:oblak} we explore the interaction between the Burge encoding and Oblak's recursive  scheme for computing $\D(P)$. This leads to a new proof (independent of Basili \cite{Bas-I}) that  Oblak's algorithm  does indeed  output $\D(P)$.

\subsection*{Related Work}

Other closely related topics have also been studied over the last several decades. For some partial results on pairs of partitions describing commuting nilpotent matrices see \cite{BDKOS, Obl-2}. Irreducibility of the variety of pairs of commuting nilpotent matrices was proved by Baranovsky \cite{Bar} (see also Basili \cite{Bas-2}). In the general Lie algebra setting, the question of irreducibility of the nilpotent commuting variety was resolved by Premet \cite{Pre}. Nilpotent and other commuting varieties of triples (and $d$-tuples, for larger $d$) of  matrices have also been considered, see e.g., \cite{Gur,GS,HaHy,Han,HoOm,JeSi,LNS,NgSi,Siv-1,Siv-2,Siv-3}. The connection between commuting nilpotent matrices and Artinian algebras was studied extensively in recent years (e.g. \cite{IMM}); see \cite{AIM-survey} for a survey of the topic. Finally, we note that results on commuting matrices have been applied to multiparameter eigenvalue problems \cite{Atk,Kos}, the area of research that originated from applications to boundary value problems for elliptic partial differential equations \cite{Vol}.

Burge's correspondence provides a natural description of all Jordan types of matrices that have equal dense or dominant orbit in the nilpotent commutator. It is an example of a combinatorial description of objects that are of great interest in algebra and geometry, and have interesting applications in physics \cite{JaMa}.

\section{Burge Codes and the Descent Map}
\label{sec:groundwork}

The goal of this section is to describe Burge's elegant bijection \cite{Bur-1} between partitions $\Pset$  and  binary words $\Wset$.  This is best expressed in terms of the frequency representation of partitions, which we introduce in \S\ref{sec:prelim}.  The correspondence itself is defined in \S\ref{sec:burge}. All results presented therein can be found (at least implicitly) in~\cite{Bur-1}, but we offer a complete exposition in the interests of fixing notation and making this article self-contained.  In \S\ref{sec:universal} we define the \emph{descent map} on $\Pset$. This key transformation arises naturally from Burge's construction and lies at the heart of our results.  Finally, in \S\ref{sec:burgeadditional}, we describe how another bijection of Burge transforms the elements of $\D^{-1}(Q)$ into partitions whose diagonal hook lengths are given by $Q$.

\subsection{Preliminaries}
\label{sec:prelim}

Recall that a \emph{partition} $P \in \Pset$  is a weakly decreasing sequence $(p_1,\ldots,p_k)$ of positive integers. The $p_i$ are called the \emph{parts} of $P$. 
%The \emph{size} and \emph{length} of $P$ are the sum $|P|:=\sum_i p_i$ and number $\len{p}:=k$ of parts, respectively.
The \emph{size} of $P$ is the sum of the parts,  $|P|:=\sum_i p_i$, and the \emph{length} of $P$ is the number of parts,   $\len{P}:=k$. The \emph{empty partition} is the unique element $\emptypartition \in \Pset$ of size (and length) 0. characterized by $|\emptypartition|=0$ or $\len{\emptypartition}=0$. Partitions can equivalently be regarded as unordered multisets of positive integers. We use the standard  multiset notation  $[1^{f_1}, 2^{f_2} \cdots]$ to denote the partition whose parts consist of $f_1$ copies of 1, $f_2$ copies of 2, etc.

The \emph{2-measure} of a partition $P$, denoted $\two{P}$, is the maximum length of a super-distinct subpartition of $P$.  For example, $P=(8,7,4,4,3,2,2,1)$ contains the subpartition $(7,4,1) \in \RRset$ of length 3 and none longer, so $\two{P}=3$.  This statistic was recently studied  by Andrews et al~\cite{ACZ} through the lens of $q$-series. Here it will play a fundamental combinatorial role.  Clearly we have $0 \leq \two{P} \leq \len{P}$, with $\two{P}=0 \iff P=\emptypartition$, and $\two{P}=\len{P} \iff P \in \RRset$.

Let $\Fset$ be the set of all finitely supported sequences of nonnegative integers.  We will work extensively with partitions via their ``frequency'' representations in $\Fset$ given by the trivial correspondence  $[1^{f_1} 2^{f_2} \cdots] \leftrightarrow (f_1,f_2,\ldots)$.  This identification of $\Fset$ and $\Pset$ should be kept in mind throughout, but when we wish to be explicit we  will write $f(P)$ for the frequency sequence of $P \in \Pset$ and  $P(f)$ for the partition corresponding to $f \in \Fset$.

\begin{convention} When working with generic $f \in \Fset$, we adopt the convention that $f_0=0$.  When working with specific $f \in \Fset$,  trailing zeroes will often be omitted; e.g. $(0,3,2,0,1) = (0,3,2,0,1,0,0,\ldots)$.
\end{convention}

We now introduce various nomenclature related to $\Fset$.  The \emph{support} of $f \in \Fset$ will be denoted  $\supp(f):=\{i \geq 1 \,:\, f_i \neq 0\}$. The size and length statistics on $\Pset$ are extended to $\Fset$ in the obvious way, namely 
$$
	\textstyle
	|f| :=  |P(f)| = \sum_i if_i, \qquad\text{and}\qquad
	\len{f} := \len{P(f)} = \sum_i f_i.
$$
Similarly we let $\two{f}:=\two{P(f)}$. Note that this is the maximum size of a subset of $\supp(f)$ that contains no consecutive pairs $\{i,i+1\}$.

A \emph{spread} of $f$ is defined to be  a maximal interval $[i,j] \subseteq \supp(f)$, and
a spread  $[i,i]=\{i\}$ of size 1 is said to be \emph{trivial}. Finally, we define  sets 
\begin{align*}
\textstyle
	\Il(f) &:= \bigcup \{i, i+2, \ldots, i+2\lfloor\tfrac{j-i}{2}\rfloor\} \\
	\Ir(f) &:= \bigcup \{j, j-2, \ldots, j-2\lfloor\tfrac{j-i}{2}\rfloor\}, 
\end{align*}
where the unions extend over all spreads $[i,j]$ of $f$. The importance of $\Il(f)$ and $\Ir(f)$ will be apparent shortly.  For now we observe only that they are  of equal size, namely $\two{f}=\sum \lceil\frac{j-i+1}{2}\rceil$. 

\begin{exmp}
Let $f=(2,1,0,3,2,2,0,0,1)$.  The corresponding partition is  $P=P(f) = [9,6^2,5^2,4^3,2,1^2]$, and $|f|=|P|=47$ and $\len{f}=\len{P}=11$.  The spreads of $f$ are $\{1,2\}$, $\{4,5,6\}$ and $\{9\}$, so  $\Il(f)=\{1,4,6,9\}$, $\Ir(f)=\{2, 4, 6, 9\}$ and $\two{f}=\two{P}=4$. 
\end{exmp}
 
\subsection{The Burge Correspondence}
\label{sec:burge}

We begin by recasting the sets $\Il(f)$ and $\Ir(f)$, defined above, in terms of certain partial pairings of the entries of $f \in \Fset$. Observe that the elements $i_1 < i_2 < \ldots < i_m$ of $\Il(f)$ specify the pairs $(f_{i_1},  f_{i_1+1})$, \ldots,  $(f_{i_m}, f_{i_m+1})$  that result from parsing $f$ from left-to-right and grouping  consecutive entries $f_i,f_{i+1}$ with $f_i > 0$. We call these the \emph{forward pairs} of $f$. Similarly, the elements $j_1 > j_2 > \ldots > j_m$ of $\Ir(f)$ determine the  \emph{backward} pairs $(f_{j_1-1},f_{j_1})$, $\ldots$, $(f_{j_m-1},f_{j_m})$ obtained by parsing $f$ from right-to-left and grouping consecutive entries $f_{j},f_{j-1}$ with $f_j > 0$. Here we rely on the convention $f_0=0$ to allow for the  degenerate backward pair $(f_0,f_1)$, which arises if and only if $1 \in \Ir(f)$.

\begin{exmp}
\label{exmp:pairing}
For $f = (2,2,1,3,1,0,4,0,0,2,1)$ we have $\Il(f)=\{1,3,5,7,10\}$ and $\Ir(f)=\{1,3,5,7,11\}$. The forward/backward pairs of $f$ are indicated below  with arrows  pointed in the direction of parsing.
$$
	\underset{\text{forward}}{{\color{red}0\,\,}(\Lpair{2,2},\Lpair{1,3},\Lpair{1,0},\Lpair{4,0},0,\Lpair{2,1})}
\qquad\qquad
	\underset{\text{backward}}{\Rpair{{\color{red}0\,\,} (2},\Rpair{2,1},\Rpair{3,1},\Rpair{0,4},0,0,\Rpair{2,1})}
$$
The fictional entry $f_0=0$ has been prepended in red. Observe that every nonzero entry of $f$ appears in one forward and one backward pair, while all unpaired entries of $f$ are 0.
\end{exmp}

Let $\AP := \{f \in \Fset \,:\, 1 \not \in \Ir(f)\}$ and $\BP := \{f \in \Fset \,:\, 1 \in \Ir(f)\}$, observing that these sets partition $\Fset$. Also make note of the equivalent characterizations  $f \in \BP \iff (f_0,f_1)$ is a backward pair $\iff$ 1 is contained in a spread of $f$ of odd size.

We now  introduce two related transformations $\map{\a}{\Fset}{\AP}$ and $\map{\b}{\Fset}{\BP}$,  along with a mapping $\map{\pt}{\Fset}{\Fset}$ that undoes them.  Each of these acts as a sequence of raising/lowering operators on the forward or backward pairs of $f$.  The precise recipes are as follows:
\begin{itemize}
\item	$\a(f)$ is obtained from $f$ by replacing $(f_{i}, f_{i+1})$ with $(f_{i}-1, f_{i+1}+1)$ for each $i \in \Il(f)$.

\item	$\b(f) := (f_1+1, \a(f_2,f_3,\ldots))$

\item	$\pt(f)$ is obtained from $f$ by replacing $(f_{j-1}, f_{j})$ with $(f_{j-1}+1, f_{j}-1)$ for each $j \in \Ir(f)$, where this is understood to mean  only that  $f_1$ is  reduced by 1 in the case $j=1$ (i.e., the fictional $f_0$ is ignored). 

\end{itemize}

In effect, $\a$ scans $f$ from left-to-right and ``promotes'' each forward pair $(f_{i}, f_{i+1})$ to $(f_{i}-1, f_{i+1}+1)$, thereby transforming the forward pairs of $f$ into the backward pairs of $\a(f)$. (That is, $i \in \Il(f) \iff i+1 \in \Ir(\a(f))$.)  In particular, for any $f \in \Fset$ we have $1 \not \in \Ir(\a(f))$ and hence $1 \in \Ir(\b(f))$, validating the claim that $\a$ and $\b$  map $\Fset$ into $\AP$ and $\BP$, respectively.  Moreover, $f$ can be recovered from either $\a(f)$ or $\b(f)$ by scanning from right-to-left and ``demoting'' backward pairs. This is exactly the action of $\pt$.

\begin{lem}
Let $\a,\b,\pt$ be defined on $\Fset$ as above.  Then $\a$ and $\b$ are bijections from $\Fset$ to $\AP$ and $\BP$, respectively, while $\map{\pt}{\Fset}{\Fset}$ is 2-to-1 and restricts to $\a^{-1}$ on $\AP$ and $\b^{-1}$ on $\BP$.
\end{lem}

\begin{exmp}
\label{exmp:burgemapping}
Let $f = (2,2,1,3,1,0,4,0,0,2,1)$ as in  Example~\ref{exmp:pairing}.  Then
\begin{align*}
\a(f)
&=\a(\Lpair{2,2},\Lpair{1,3},\Lpair{1,0},\Lpair{4,0},0,\Lpair{2,1})
=(1,3,0,4,0,1,3,1,0,1,2) \\[1.5ex]
	\b(f) &= (2+1, \a(\Lpair{2,1},\Lpair{3,1},0,\Lpair{4,0},0,\Lpair{2,1}))
	 = (3, 1, 2, 2, 2, 0, 3, 1, 0, 1, 2) \\[1.5ex]
	 \pt(f)&=
	\Rpair{\pt(2},\Rpair{2,1},\Rpair{3,1},   		
	\Rpair{0,4},0,0,\Rpair{2,1}) 
	= (1,3,0,4,0,1,3,0,0,3,0).
\end{align*}
Forward/backward pairs are promoted/demoted by transferring one unit along each arrow. For the degenerate backward pair $(f_0,f_1)$, which corresponds to $1 \in \Ir(f)$ (i.e., $f \in \BP$), demotion by $\pt$ amounts to reducing $f_1$ by 1.
We leave it to the reader to verify that $\pt(\a(f))=\pt(\b(f))=\b(\pt(f))=f$.  Note, however, that 
$\a(\pt(f))=(0,4,0,3,1,0,4,0,0,2,1) \neq f$. 
This inequality is due to the fact that $f \not \in \AP$.
 
%$$
%\a(\pt(f))=\a(\Lpair{1,3},0,\Lpair{4,0},\Lpair{1,3},0,0,\Lpair{3,0})
%=(0,4,0,3,1,0,4,0,0,2,1) \neq f,
%$$
\end{exmp}

%It is straightforward to verify that $\pt(\a(f))=\pt(\b(f))=f$ for all $f$, while $\a(\pt(f))=f$ for $f \in \AP$ and $\b(\pt(f))=f$ for $f \in \BP$.   

Clearly we have $0 \leq |\pt(f)|<|f|$ for all $f \neq \emptyf$. Applying $\pt$ repeatedly to any $f \in \Fset$ will therefore  result in $\pt^{k}f=\emptyf$ for some least positive integer $k \geq 1$. This leads to the following  definitions.

\begin{defn}
We define the \emph{Burge chain} of $f\in \Fset$ to be the sequence
$\pti{f}{0}, \pti{f}{1}, \pti{f}{2}, \ldots, \pti{f}{k}$, where $k$ is
 the smallest positive integer satisfying $\pti{f}{k}=\emptyf$.
%We refer to the length of this sequence (namely $k+1$) as the \emph{Burge length} of $f$. Suppose $f \in \Fset \setminus\{\emptyf\}$ has Burge length $n$.  
The  \emph{Burge code} of $f$ is the binary word $\burge{f} := \w_1 \w_2 \cdots \w_{k+1} \in \{\a,\b\}^*$ defined by 
%$\w_i = \a$ if $\pti{f}{i-1} \in \AP$ and $\w_i = \b$ if $\pti{f}{i-1} \in \BP$.
$$
	\w_i = \begin{cases}
	\a & \text{if $\pti{f}{i-1} \in \AP$,} \\
	\b & \text{if $\pti{f}{i-1} \in \BP$.}
	\end{cases}
$$ 
\end{defn}

So $f=\emptyf$ has trivial Burge chain $\emptyf$ and Burge code $\burge{\emptyf}=\a$, whereas the  chains/codes of all other sequences $f$ are of length at least 2. In fact since $\pt f = \emptyf$ only for $f = \emptyf$ and $f=(1)$, the chain of every $f \neq \emptyf$  ends with $\pti{f}{k-1} = (1) \in \BP$ and $\pti{f}{k} = \emptyf \in \AP$, and thus the code $\burge{f}$ ends with $\w_{k}\w_{k+1}=\b\a$.

Since $\pt|_{\AP}=\a^{-1}$ and $\pt|_{\BP}=\b^{-1}$, the definition of $\burge{f}=\w_1\cdots\w_n$   ensures that $\pti{f}{i-1}= (\w_i \circ \pt)(\pti{f}{i-1})=\w_i(\pti{f}{i})$ for all $i$. Thus we have
$$
f =\w_1(\pti{f}{1})=\w_1\w_2(\pti{f}{2})=\w_1\w_2\w_3(\pti{f}{3})=\cdots=(\w_1\w_2\cdots\w_n)(\emptyf),
$$
where the products of  $\w_i$ are to be interpreted as functional composition in the usual right-to-left order.  
That is, the right-to-left reading of $\burge{f}$ specifies the unique manner by which $f$ can be ``built'' from $\emptyf$ via iterative applications of $\a$ and $\b$, beginning with a single application of $\a$.

\begin{prop}
The Burge encoding $f \mapsto \burge{f}$ is a one-one correspondence between $\Fset$ and the set $\Wset:=(\a^*\b)^*\a$ of all finite words on  $\{\a,\b\}$ that end with a singleton $\a$.
\end{prop}

\begin{exmp}
\label{exmp:burge}
Let $f=(1,2,1,0,1)$. The Burge chain of $f$ is displayed below, along with the values of $\w_i$ along the chain.  The Burge code $\burge{f}=\b\a\b\a\a\a\b\b\b\a$ is obtained by concatenating these values.
$$
\begin{array}{c|l|l}
i & \pti{f}{i} & \w_{i+1} \\
\hline
0 & (1,2,1,0,1) & \b \\
1 & (0,3,0,1) & \a \\
2 & (1,2,1) & \b \\
3 & (0,3) & \a \\
4 & (1,2) & \a \\
5 & (2,1) & \a \\
6 & (3) & \b \\
7 & (2) & \b \\
8 & (1) & \b \\
9 & \emptyf & \a
\end{array}
$$
Conversely, given the word $\w=\a\b\b\a\b\a \in \Wset$ we  build its corresponding partition by iterating $\a$ and $\b$ as follows: 
$$
\emptyf \overset{\a}{\mapsto} 
\emptyf \overset{\b}{\mapsto} 
(1) \overset{\a}{\mapsto} 
(0,1) \overset{\b}{\mapsto} 
(1,0,1) \overset{\b}{\mapsto}
(2,0,0,1)  \overset{\a}{\mapsto} (1,1,0,0,1).$$
Thus $\w$ is the Burge code of $f=(1,1,0,0,1)$. 
\end{exmp}

We now describe how size, length and 2-measure are affected by $\pt$, and use this to express $|f|, \len{f}$ and $\two{f}$  in terms of familiar statistics on $\burge{f}$.  For convenience we  hereon  omit parentheses and write $\pt f$ instead of $\pt(f)$. 

\begin{lem}
\label{lem:stats}
For any $f \in \Fset$ we have: 
\begin{enumerate}
\item	$\displaystyle
	\len{\pt f} =
	\begin{cases} 
	\len{f} - 1 &\text{if $f \in \BP$} \\
	\len{f} &\text{otherwise.}\\
	\end{cases}$
\item	$\displaystyle |\pt f| = |f| - \two{f}$

\item 	$\displaystyle
	\two{\pt f} =
	\begin{cases}
	\two{f} - 1 &\text{if $f \in \BP$ and $\pt f \in \AP$} \\
	\two{f} &\text{otherwise.}
	\end{cases}$
\end{enumerate}
\end{lem}

\begin{proof}
The sum $\len{f}=\sum_i f_i$ is invariant under the demotion $(f_{j-1},f_j) \rightsquigarrow (f_{j-1}+1,f_j-1)$ except when $j=1$, in which case it is decreased by 1, whereas $|f|=\sum_i if_i$ is decreased by 1 by every such demotion. Claims (1) and (2) follow.

For (3), first observe that if $f \in \AP$ then $\pt$ maps the backward pairs of $f$  to the forward pairs of $\pt f$, i.e., $i \in \Ir(f) \iff i-1 \in \Il(\pt f)$. Hence $\two{\pt f}=|\Il(\pt f)|=|\Ir(f)|=\two{f}$ in this case.

Now suppose $f \in \BP$ and let $g = (f_2,f_3,\ldots)$. Then $g \in \AP$ and   $\pt f = (f_1-1, \pt g)$.   Interpreting $\two{\cdot}$ as a count of backward pairs gives
$$
\two{f}=1+\two{g} \qquad\text{and}\qquad
	\two{\pt f}  
	= \begin{cases}
		\two{\pt g} &\text{if $\pt f \in \AP$} \\
		1+\two{\pt g} &\text{if $\pt f \in \BP$}.
	\end{cases}
$$
But we already established that $g \in \AP$ implies $\two{g}=\two{\pt g}$, so the above identities prove (3) in the case $f \in \BP$. 
\end{proof}

Let $\w=\w_1\w_2\cdots\w_n \in \{\a,\b\}^*$. The \emph{descent set} of $\w$, denoted $\Des{\w}$, is the set of all indices $i$ for which $\w_i\w_{i+1}=\b\a$. That is, $\Des{\w}$ records the positions of descents in $\w$ relative to the ordering  $\b > \a$.   The \emph{descent number} and \emph{major index} of $\w$ are defined by $\des{\w}:=|\Des{\w}|$ and $\maj{\w}:=\sum_{i \in \Des{\w}} i$, respectively.

\begin{prop}
\label{prop:stats}
Let  $f \in \Fset$ and let $\w = \burge{f}$.  Then: 
\begin{enumerate}
\item	$\len{f} = $ \#\,occurrences of $\b$ in $\w$
\item	$|f|=\maj{\w}$
\item	$\two{f} = \des{\w} =$ \#\,occurrences of $\b\a$ in $\w$
\end{enumerate}
\end{prop}

\begin{exmp}
In Example~\ref{exmp:burge} we saw that $f=(1,2,1,0,1)$ has Burge code $\w =\b\a\b\a\a\a\b\b\b\a$. Note that $\w$ contains $5=\len{f}$ copies of $\b$, and we have $\Des{\w}=\{1,3,9\}$,   $\maj{\w}=13=|f|$ and  $\des{\w}=3=\two{f}$. 
\end{exmp}

\begin{proof}[Proof of Proposition~\ref{prop:stats}:]
Claims (1) and (3) are immediate from the corresponding parts of Lemma~\ref{lem:stats}. To prove (2), let $\w=\w_1\w_2\cdots\w_{k+1}$ and $\Des{\w} = \{i_1, i_2, \ldots, i_m \}$, where  $1 \leq i_1 < i_2 < \cdots < i_m=k$.  Iterating part 2 of the lemma gives $|f|=\sum_{r=0}^{k} \two{\pti{f}{r}}$, while iterating part 3 gives $\two{\pti{f}{r}}=m-j$ for $r$ in the range $i_j \leq r < i_{j+1}$, where we take $i_0=0$ and $i_{m+1}=\infty$. Hence $|f| = \sum_{j=0}^{m-1} (m-j)(i_{j+1}-i_j) = \sum_{j=0}^{m-1} i_{j+1} = \maj{\w}$.
\end{proof}

\subsection{The Descent Map}
\label{sec:universal}

We have described a three-way equivalence between partitions, frequency vectors, and binary words, namely
$$
	\Pset \xrightarrow{\text{frequency $f$}}
	\Fset \xrightarrow{\text{Burge code $\Omega$}}
	\Wset.
$$
With this in mind we adopt the following convention.

\begin{convention}
All notation defined previously for elements of $\Fset$ and $\Wset$  is extended to $\Pset$ through $f(\cdot)$ and $\burge{\cdot}$. In particular, we define $\pt P := \pt(f(P))$, $\pti{P}{i}:=\pti{f(P)}{i}$, $\burge{P}:=\burge{f(P)}$ and then $\Des{P}:=\Des{\burge{P}}$. 
\end{convention}

Suppose  $\Des{P}=\{i_1,i_2,\ldots,i_m\}$ where $i_1 < i_2 < \cdots < i_m$.  Then evidently $i_j-i_{j-1} \geq 2$ for all $j$, and   Proposition~\ref{prop:stats}(2) gives $\sum_j i_j = \maj{\burge{P}}=|P|$.  Thus $\Des{P}$ can be regarded as a super-distinct partition of size $|P|$.  Moreover, if $\burge{P}=\w_1\w_2\cdots\w_n$ then $\burge{\pt P} = w_2\cdots\w_n$, so we have $\Des{\pt P} = \{i_1-1,i_2-1,\ldots,i_m-1\} \setminus \{0\}$.  See, for example, Figure~\ref{fig:ptexample}.

\begin{figure}
$$
\begin{array}{l|l|l|l|l}
i & 
\pti{f}{i} &
\burge{\pti{f}{i}} &
\pti{P}{i} & 
\Des{\pti{P}{i}}
\\
\hline
0 &  (1,1,0,1,0,0,1) & \a\b\a\b\b\a\b\a & [7,4,2,1] & [7,5,2] \\
1 &  (2,0,1,0,0,1) & \b\a\b\b\a\b\a & [6,3,1^2] &[6,4,1] \\
2 &  (1,1,0,0,1) & \a\b\b\a\b\a & [5,2,1] & [5,3] \\
3 &  (2,0,0,1) & \b\b\a\b\a & [4,1^2] & [4,2] \\
4 &  (1,0,1) & \b\a\b\a & [3,1] & [3,1] \\
5 &  (0,1) & \a\b\a & [2] & [2] \\
6 &  (1) & \b\a & [1] & [1] \\
7 &  \emptyf & \a & \emptypartition & \emptypartition 

\end{array}
$$
 \caption{The evolution of $\Des{\pti{P}{i}}$ along a Burge chain.}
 \label{fig:ptexample}
 \end{figure}

\begin{defn}
For a partition $P$, let $P-1$ denote be the \emph{reduced partition} obtained by subtracting 1 from each part of $P$ and eliminating any resulting zeros. % For example, $P=[6,4^3,2^4,1^3]=[5,3^3,1^4]$. 
\end{defn}
 
We have just seen that the \emph{descent map}
 $P \mapsto \Des{P}$ is a size-preserving function from $\Pset$ to $\RRset$ satisfying $\Des{\pt P}=\Des{P}-1$.  It is easy to see  that this is the unique such function.

\newcommand{\genop}{\mathfrak{Q}}

\begin{lem}
\label{lem:uniqueness} Let $\Pset'\subseteq\Pset$ be a nonempty set of partitions such that for every $P\in\Pset'$, $\pt P\in \Pset'$. 
Suppose $\map{\genop}{\Pset'}{\Pset}$ satisfies
\begin{enumerate}
\item $|\genop(P)|=|P|$ and 
\item $\genop(\pt P) = \genop(P)-1$. 
\end{enumerate}
Then $\genop(P) = \Des{P}$ for all $P \in \Pset'$.   
\end{lem}

\begin{proof}
Condition (1) forces $\genop(\emptypartition)=\emptypartition$, and if $\genop(\pt P)$ has frequencies $(f_1,f_2,\ldots)$ then  (1) and (2) force $\genop(P)$ to have frequencies $(|P|-|\pt P|-\sum_i f_i, f_1, f_2, \ldots)$.   So $\genop(P)$ is determined uniquely by induction on $|P|$.
\end{proof}

 Theorem~\ref{thm:main} is equivalent to statement that $\D(P) = \Des{P}$ for all $P \in \Pset$.  In Section~\ref{sec:main} we prove this identity by showing that $\D$ satisfies the hypothesis of Lemma~\ref{lem:uniqueness}.
 
We now present a multi-faceted characterization of super-distinct partitions.  Observe that the equivalence $P \in \RRset \Leftrightarrow \Des{P}=P$ together with Theorem~\ref{thm:main}  imply that $P \in \RRset$ if and only if $\D(P) = P$.

\begin{prop}
\label{prop:characterization}
Let $P \in \Pset$ and $f=f(P)$. The following are equivalent:
\begin{enumerate}
\item	$P \in \RRset$
\item	$\len{P}=\two{P}$
\item	$\burge{P}$ does not contain the substring $\b\b$
\item	$f_i=1$ if $i \in \Ir(f)$ and $f_i=0$ otherwise
\item	$\pt f = (f_2,f_3,\ldots)$
\item	$\pt P = P-1$
\item	$\Des{P}=P$
\end{enumerate}

\end{prop}

\begin{proof}
Equivalences $(1)\Leftrightarrow(2)$, $(1) \Leftrightarrow (4)$ and $(5)\Leftrightarrow(6)$ are clear from the respective definitions, while $(2)\Leftrightarrow(3)$ is immediate from the fact that $\len{P}$ and $\two{P}$ count occurrences of $\b$ and $\b\a$ in $\burge{P}$. (See Proposition~\ref{prop:stats}.)

%Have $P \in \RRset$ iff $f_i \in \{0,1\}$ for all $i$, with no consecutive nonzero. Given such $f$, we have $\Ir(f)=\{i \,:\, f_i=1\}=\supp(f)$.
If $f$ satisfies (4) then by definition of $\pt$  we have $(\pt f)_i = 1$ if $i+1  \in \Ir(f)$ and $(\pt f)_i=0$ otherwise. Thus $(4) \Rightarrow (5)$. Conversely, suppose $\pt f = (f_2,f_3,\ldots)$. Say   $\Ir(f)=\{i_1,\ldots,i_k\}$ with $i_1 > \cdots > i_k$.
%Let $\Ir(f)=\{i_1,\ldots,i_k\}$ with $i_1 > i_2 > \cdots > i_k$. 
Then $(f_{i_j-1}+1,f_{i_j}-1)=(f_{i_j}, f_{i_j+1})$ for all $j$. Since $f_{i_1+1}=0$ we have $(f_{i_1-1},f_{i_1})=(0,1)$, which in turn implies $f_{i_2+1}=0$ and therefore $(f_{i_2-1},f_{i_2})=(0,1)$, etc.  Thus $(5) \Rightarrow (4)$.

We know $\Des{P} \in \RRset$ for all $P$, so clearly $(7)\Rightarrow(1)$. The remaining implication, $(1) \Rightarrow (7)$, follows by applying Lemma \ref{lem:uniqueness} to 
$\map{\genop}{\RRset}{\Pset}$ given by
$\genop(T)=T$.
Note that $\RRset$ is $\pt$-invariant and that this function satisfies the conditions of the lemma due to the fact that, as proven above, (1) implies (6).  
%To complete the proof we use induction on $|P|$ to  show  $(1)$ and $(6)$ imply (7).  Clearly the implication holds for $P=\emptypartition$.  Suppose $P\in \RRset \setminus\{\emptypartition\}$. Certainly $P-1 \in \RRset$ and $|P-1| < |P|$, so the induction hypothesis and (6) yield $P-1=\Des{P-1}=\Des{\pt P}$. But  we know $\Des{\pt P}= \Des{P}-1$, so we have $P-1 = \Des{P}-1$.  Since $|P|=|\Des{P}|$ this forces $P=\Des{P}$, as desired. 
\end{proof}

\subsection{Additional Remarks}  
\label{sec:burgeadditional}
Burge's work \cite{Bur-1,Bur-2} provides additional bijections between partitions and $\{a,\b\}$-words. One of these~\cite[\S2]{Bur-1} interprets $\a$ and $\b$ as operators on $\Pset$ that enlarge the diagonal hooks of a partition (viewed as a Young diagram or Ferrer's graph) along the horizontal and vertical axes, respectively. Together with Corollary~\ref{cor:boxconjecture}, it is easy to check that this gives a correspondence between $\D^{-1}(Q)$ and the set of partitions with hook lengths $q_1,\ldots,q_k$ (and hence Durfee square of size $k=\len{Q}$). Such a mapping was sought by the proposers of the Box Conjecture, who inferred its existence enumeratively; see~\cite[Rmk. 4.12]{IKvSZ}.
Curiously, this correspondence also answers a question raised recently by Andrews et al \cite{ABD,ACZ}, which asks for a size- and length-preserving bijection on $\Pset$ that maps 2-measure to size of Durfee square.

It is worth noting that the aforementioned correspondence can also be defined in terms of another well-known combinatorial mapping on words.
   Let $M$ be a finite multiset with elements from an ordered alphabet $S$, and let $\mathfrak{S}(M)$ be the set of all permutations of $M$ --- equivalently, words on $S$ with content $M$.   Foata's \emph{second fundamental transformation} is a recursively-defined bijection $\map{\phi}{\mathfrak{S}(M)}{\mathfrak{S}(M)}$ that maps  major index to inversion number; i.e., $\maj{w}=\mathrm{inv}(\phi(w))$.  When the alphabet $S$ is of size 2, e.g. $S=\{\a,\b\}$ with $\b > \a$,  Foata's transformation has a simple non-recursive description; see~\cite[Prop. 2.2]{Sag-Sav}.  Let $Q =(q_1,\ldots,q_k) \in \RRset$ and suppose $P \in \D^{-1}(Q)$ has Burge code $w$ given by~\eqref{eq:code}. Then one finds
    $$
    \phi(w) = \b \a^{\delta_k-i_k} \b \a^{\delta_{k-1}-i_{k-1}}
    \cdots \b \a^{\delta_2-i_2} \b \a^{\delta_1-i_1} 
    \b^{i_1-1}\a \b^{i_2-1}\a \cdots \b^{i_k-1}\a.
    $$
    This word describes a lattice path by reading from right-to-left and interpreting $\a$ and $\b$ as  east and north steps, respectively. This path, in turn, determines a partition $P'$ by tracing the boundary of its Young diagram. By construction we have $|P'|=\mathrm{inv}(\phi(w))=\maj{w}=|P|$ and $\len{P'}=\len{P}=\sum_{j=1}^k i_j$, and it is easy to check that $P'$ has diagonal hooks of lengths $q_1,\ldots,q_k$.  Indeed $P \mapsto P'$ is precisely the bijection described in the previous paragraph.  

Burge's bijections \cite{Bur-1,Bur-2} are in fact more refined than the versions presented here, as they track certain  additional partition  statistics that were not relevant for our purposes. However, some of these statistics may have interesting algebraic manifestations and  be of interest in the future when working on the question of nilpotent orbits corresponding to commuting matrices. In the more general setting it is fruitful to view the associated binary words as lattice paths, as this affords geometric interpretations of the additional parameters  \cite{AB, Bre, JaMa}.

\section{The Burge Correspondence and Invariant Subspaces}
\label{sec:main}

In this section we first define the maximal commutative subalgebra $\U_B$ of the \emph{nilpotent commutator} $\Nil_B=\{A\in M_n(F);\ AB=BA,\ A\textrm{ nilpotent}\}$. The image $W$ of a generic element of $\U_B$ is an invariant subspace for $B$. It is of dimension $n-\mu_2(P)$ \cite[Prop. 2.4]{Bas-2}. Our main result of the section is the fact that if $B$ is of Jordan type $P \in \Pset$, then the Jordan type of the restriction $B|_W$ is $\pt P$. Theorem~\ref{thm:main} is an immediately corollary due to the universality of the descent map (Lemma~\ref{lem:uniqueness}).

\subsection{The Nilpotent Commutator and its Maximal Nilpotent Subalgebra}
\label{sec:nilpotentcomm}

Assume that $P=P(B)$ is the Jordan type of $B$ and that $B\in M_n(F)$ is in the (upper-triangular) Jordan canonical form. Let $f(P)$ be the frequency sequence of $P$. So $P=[1^{f_1},2^{f_2},\ldots,z^{f_z}]$, where $z$ is the size of the largest Jordan block in $B$. Thus $f_z\neq 0$. Recall that $\supp(f)=\{i\in\mathbb{N};\, f_i\neq 0\}$ is the support of $f$. We then write also $P=[i^{f_i}]_{i\in\supp(f)}$. The Jordan canonical form induces a cyclic decomposition of the underlying vector space 
$$V=\bigoplus_{i\in\supp(f)}\bigoplus_{k=1}^{f_i} V_{ik}$$ 
into cyclic subspaces $V_{ik}$ for $B$. We denote by $v_{ik1}$ the cyclic generator of $V_{ik}$. Then $v_{ikh}=B^{h-1}v_{ik1}$, $h=1,2,\ldots,i$, is a Jordan chain for $B$ spanning $V_{ik}$ and the union of all these Jordan chains
$$\Ba_B=\bigcup_{i\in\supp(f)}\bigcup_{k=1}^{f_i}\{v_{iki},v_{ik,i-1}\ldots,v_{ik1}\}$$ 
is a Jordan basis for $B$. Note that $Bv_{iki}=0$ for all $i$ and $k$.

The following lemma gives a possible modification of Jordan basis $\Ba_B$. 

\begin{lem}\label{new_Jordan_basis}
    Choose $s\in\supp(f)$ and $t\in\{1,2,\ldots,f_s\}$. Let $1\le r\le s-1$. Then for any scalars $\alpha_{ikj}$, $i\in\supp(f)$, $i\le s-r$, $k\in\{1,2,\ldots,f_i\}$, $j\in\{1,2,\ldots,i\}$, vectors 
    $$v'_{sth}=B^{h-1}\left(v_{str}+\sum_{i\in\supp(f),\,i\le s-r}\sum_{k=1}^{f_i}\sum_{j=1}^i\alpha_{ikj} v_{ikj}\right),\ h=1,2,\ldots,s-r+1,$$ 
    form a Jordan chain for $B$. For $r=1$, we obtain another Jordan basis for $B$ on $V$ by replacing $\{v_{st1},v_{st2},\ldots, v_{sts}\}$ with $\{v'_{st1},v'_{st2},\ldots, v'_{sts}\}$ in $\Ba_B$.
\end{lem}

\begin{proof}
    Use the fact that $\Ba_B$ is a Jordan basis for $B$ and the restrictions imposed on $r$ and $i$.
\end{proof}

It is well-known (see e.g. \cite[\S 1.3]{SuTy}) that any matrix $A$ commuting with $B$ (i.e., a matrix $A$ in the \emph{commutator} $\Cm_B$ of $B$) has two-step block structure: First with respect to the sizes of Jordan blocks into $A=\left[A_{ij}\right]_{i,j\in\supp(f)}$, and then with respect to the frequency of each size to $A_{ij}=\left[A^{kl}_{ij}\right]_{k,l=1}^{f_i,f_j}$. Each block $A_{ij}^{kl}$ is an upper-triangular Toeplitz matrix either of the form
\begin{equation}\label{Toeplitz-1} 
\left[\begin{array}{ccccccc}
0     &\cdots&0     &a_1^{kl}&a_2^{kl}&\cdots&a_i^{kl}\\
\vdots&      &      &0       &a_1^{kl}&\ddots&\vdots  \\
\vdots&      &      &      &\ddots    &\ddots&a_2^{kl}\\
0     &\cdots&\cdots&\cdots&\cdots    &0     &a_1^{kl}
\end{array}\right]
\end{equation}
if $i< j$, or of the form
\begin{equation}\label{Toeplitz-2}
    \left[\begin{array}{cccc}
a_1^{kl}&a_2^{kl}&\cdots&a_j^{kl}\\
0       &a_1^{kl}&\cdots&\vdots  \\
\vdots  &\ddots  &\ddots&a_2^{kl}\\
\vdots  &        &0     &a_1^{kl}\\
\vdots  &        &      &0       \\
0     &\cdots&\cdots&0     
\end{array}\right]
\end{equation}
if $i>j$. When $i=j$ the initial columns of zeros in the first form are omitted (or equivalently the trailing rows of zeros in the second form are omitted). The parameters $a_h^{kl}$ in different blocks are all distinct. We omitted the indices $i$ and $j$ to simplify notation. 

For each $i\in\supp(f)$, observe that all the blocks in $A_{ii}$ are $i\times i$ square blocks. We choose the diagonal element of each of the blocks in $A_{ii}$ and we form new matrix $A_{i}^D=\left[a_1^{kl}\right]_{k,l=1}^{f_i}$.  Next, we point out that for each $A$ the blocks $A_i^D$, $i\in\supp(f)$, form the semisimple part of $A$ considered as an element of algebra $\Cm_B$ \cite[Lem. 2.2]{BIK}. Basili \cite[Lem. 2.3]{Bas-2} showed that an element $A\in\Cm_B$ is nilpotent (i.e., it belongs to the nilpotent commutator $\Nil_B$) if and only if its blocks $A_i^D$ are nilpotent for all $i\in\supp(f)$ and that for each $A\in\Nil_B$ it is possible to find a Jordan basis $\Ba_B$ for $B$ such that the corresponding matrices $A_i^D$ are all lower-triangular. Thus each matrix in the nilpotent commutator $\Nil_B$ is similar to a matrix in a maximal nilpotent subalgebra $\Uc_B$ of $\Nil_B$ defined by
$$\Uc_B=\{A\in\Nil_B;\ \textrm{blocks }A_i^D\ \textrm{are strictly lower-triangular for all } i\in\supp(f)\}.$$
Let us now consider an example.

\begin{exmp}
    Assume that $B$ has Jordan form determined by partition $P=[4^2,3,2^2]$, so that $f(P)=(0,2,1,2)$. A matrix $A$ in $\Uc_B$ is of the form
    $$\left[\begin{array}{ccccccccccccccc}
    0&a_{11}&a_{12}&a_{13}&0&a_{21}&a_{22}&a_{23}&b_1&b_2&b_3&c_{11}&c_{12}&c_{21}&c_{22}\\
    0&0     &a_{11}&a_{12}&0&0     &a_{21}&a_{22}&0  &b_1&b_3&0  &c_{11}&0&c_{21}\\
    0&0&0&a_{11}&0&0&0&a_{21}&0&0&b_1&0&0&0&0\\
    0&0&0&0&0&0&0&0&0&0&0&0&0&0&0\\
    a_{30}&a_{31}&a_{12}&a_{33}&0&a_{41}&a_{42}&a_{43}&d_1&d_2&d_3&e_{11}&e_{12}&e_{21}&e_{22}\\
    0     &a_{30}&a_{31}&a_{32}&0&0     &a_{41}&a_{42}&0  &d_1&d_2&0  &e_{11}&0&e_{21}\\
    0     &   0  &a_{30}&a_{31}&0&0     &0     &a_{41}&0  &0  &d_1&0  &0 &0&0 \\
    0     &0     &0     &a_{30}&0&0&0&0&0&0&0&0&0&0&0\\
    0     &f_{1} &f_{2} &f_{3} &0&g_{1}&g_{2}&g_{3}&0&h_1&h_2&i_{11}&i_{12}&i_{21}&i_{22}\\
    0     &0     &f_{1} &f_{2} &0&0    &g_{1}&g_2  &0&0  &h_1&0  &i_{11}&0&i_{21}\\
    0     &0     &0     &f_{1} &0&0    &0    &g_{1}&0&0&0&0&0&0&0\\
    0     &0     &j_{11}&j_{12}&0&0&k_{11}&k_{12}&0&l_{11}&l_{12}&0&m_{11}&0&m_{21}\\
    0     &0     &0     &j_{11}&0&0&0&k_{11}&0&0&l_{11}&0&0&0&0\\
    0     &0     &j_{21}&j_{22}&0&0&k_{21}&k_{22}&0&l_{21}&l_{22}&m_{30}&m_{31}&0&m_{41}\\
    0     &0     &0     &j_{21}&0&0&0&k_{21}&0&0&l_{21}&0&m_{30}&0&0
    \end{array}\right].$$
    Then we have
    $$A^D_4=\left[\begin{array}{cc}
        0 & 0 \\
        a_{30} & 0
    \end{array}\right],\ A^D_3=[0]\ \textrm{and}\  
    A^D_2=\left[\begin{array}{cc}
        0 & 0 \\
        m_{30} & 0
    \end{array}\right].$$
\end{exmp}

\vskip 5pt

\subsection{Invariant Subspaces}

The structure of invariant subspaces of a nilpotent matrix in Jordan form was described by Shayman \cite{Sha-1,Sha-2}.
The following proposition is an easy consequence of Shayman's results. 

\begin{prop}\label{W ia an image}
Each invariant subspace of $B$ is equal to the image of an element of $\Cm_B$. 
\end{prop}

\begin{proof}
   Suppose that $W$ is an invariant subspace of $B$ of dimension $d$. By \cite[Prop. 6]{Sha-1} it is equal to the span of columns of an $n\times d$ matrix $\widetilde{A}$ of rank $d$ in Toeplitz block form. (Note a misprint in the formulation of \cite[Prop. 6]{Sha-1} -- instead of $M$ it should be $S$.) Let us describe Shayman's result in further details. A partition $P'=(p_1',p_2',\ldots, p'_l)$ of $d$ is \emph{compatible with} partition $P=(p_1,p_2,\ldots,p_l)$ of $n$ if $0\le p_i'\le p_i$ for all $i$. Here $P$ is the Jordan type of $B$ (written with possible repetitions to conform more closely with Shayman's notation). Then $\widetilde{A}=[\widetilde{A}_{ij}]_{i,j=1}^{p_i,p_j'}$ is a block matrix and each block $A_{ij}$ is in the upper-triangular Toeplitz form (\ref{Toeplitz-1}) if $p_i\le p'_j$ or (\ref{Toeplitz-2}) if $p_i\ge p_j'$. Next, we complete the matrix $\widetilde{A}$ to an $n\times n$ matrix $A=[A_{ij}]_{i,j=1}^{p_i,p_j}$ by adding columns of zeros. If $p'_j<p_j$ then we append $p_j-p'_j$ columns of zeros on the left-side of $\widetilde{A}_{ij}$ to obtain $A_{ij}$ for each $i$. If $p_j'=p_j$ then $A_{ij}=\widetilde{A}_{ij}$ for each $i$. Observe now that $W=\im A=\im\widetilde{A}$ and that $A\in\Cm_B$ by construction of $A$.
\end{proof}

\begin{rem}
    Note that the above proof is valid over any field $F$. Matrix $\widetilde{A}$ is a full rank solution of the linear matrix equation $BX=XJ$, where $B$ is in the Jordan canonical form corresponding to $P$ and $J$ in the Jordan canonical form corresponding to $P'$. 
    
    The statement of Proposition \ref{W ia an image} holds more generally even for matrices that are not nilpotent. An invariant subspace of a matrix has a direct sum decomposition with respect to the primary decomposition (sometimes called also rational canonical form) of the underlying vector space. See e.g. \cite[Thm. 2.1.5]{GLR} for a proof over complex numbers and \cite[Thm. 4.3]{GrKo} for a proof over general $F$. Note that when the field is algebraically closed the primary decomposition coincides with the spectral decomposition \cite[Cor. 4.4]{GrKo}.
\end{rem}

\vskip 5pt

\subsection{Elementary Cyclic Column Operations}
\label{sec:elementary}

Now, consider the image $W=\im A$ of a generic element $A$ of $\Uc_B$. Since $A$ and $B$ commute it follows easily that $W$ is an invariant subspace of $B$. To determine the Jordan type of the restriction $B|_W$ we use \cite[Lem. 4]{Sha-1}. Let us briefly outline the construction. 

We denote by $\Gc_B$ the subgroup of invertible matrices in the commutator $\Cm_B$ that have all $A_{i}^D$ blocks lower-triangular, i.e., 
$$\Gc_B=\{A\in\Cm_B\cap \GL_n(F);\ \textrm{blocks }A_i^D\ \textrm{are lower-triangular for all } i\in\supp(f)\}.$$ 
For each $A\in\Uc_B$ and $Z\in \Gc_B$ it follows that $\im A = \im AZ$. Next we introduce certain \emph{elementary cyclic column operations} on matrices from $\Uc_B$ partitioned into $A_{ij}^{kl}$ blocks as per the cyclic structure of $B$. These are operations on columns of a matrix in $\Uc_B$ that generalize the standard elementary column operations on a matrix.
\begin{enumerate}
    \item[(\textit{i})] Fix $s\in\supp(f)$ and $t\in\{1,\ldots,f_s\}$. Multiply each of the elements in blocks $A_{is}^{kt}$, $i\in\supp(f)$, $k=1,\ldots,f_i$, by a nonzero scalar $\alpha$.
    \item[(\textit{ii})] Fix $s,s'\in\supp(f)$, $t\in\{1,\ldots,f_s\}$ and $t'\in\{1,\ldots,f_{s'}\}$. Take the first $a$ columns of a block $A_{is}^{kt}$, multiply each of them by a nonzero scalar $\alpha$ and add them consecutively to the last $a$ columns of block $A_{is'}^{kt'}$ for all $i\in\supp(f)$ and $k=1,\ldots,f_i$. Here we assume that $a\leq \min\{ s, s'\},$ and if $(s,t)=(s',t')$ then $a < s$.
\end{enumerate}
These elementary cyclic column operations can be realized by matrix multiplication of elements in $\Uc_B$ by certain \emph{elementary matrices} that belong to $\Gc_B$.
\begin{enumerate}
    \item[(\textit{i})] Blocks $Z_{ij}^{kl}$ of an elementary matrix $Z(s,t,\alpha)$ are given by:
    \begin{itemize}
    \item $Z_{ij}^{kl}=0$ if $(i,k)\neq (j,l)$,
    \item $Z_{ii}^{kk}=I$ if $(i,k)\neq (s,t)$,
    \item $Z_{ss}^{tt}=\alpha I$.
    \end{itemize}
    \item[(\textit{ii})] If $(s,t)\neq (s',t')$ then blocks $Z_{ij}^{kl}$ of an elementary matrix $Z(s,t,s',t',\alpha)$ are given by:
    \begin{itemize}
    \item $Z_{ij}^{kl}=0$ if $(i,k)\neq (s,t)$ or $(j,l)\neq (s',t')$,
    \item $Z_{ii}^{kk}=I$,
    \item in $Z_{ss'}^{tt'}$ the entries on the $a$-th diagonal (counting from the top-right corner of the block) are equal to $\alpha$ and other entries are all equal to $0$.
    \end{itemize}
    If $(s,t)= (s',t')$ then blocks $Z_{ij}^{kl}$ of an elementary matrix $Z(s,t,s,t,\alpha)$ are given by:
    \begin{itemize}
    \item $Z_{ij}^{kl}=0$ if $(i,k)\neq (s,t)$ or $(j,i)\neq (s,t)$,
    \item $Z_{ii}^{kk}=I$ if $(i,k)\neq (s,t)$,
    \item in $Z_{ss}^{tt}$ the entries on the $a$-th diagonal (counting from the top-right corner of the block) are equal to $\alpha$, the diagonal entries are equal to $1$ and all other entries are equal to $0$.
    \end{itemize}
\end{enumerate}

Choose a generic matrix $A\in\Uc_B$ partitioned into blocks $A_{ij}^{kl}$. All blocks with fixed pair of indices $(i,k)$ form a \emph{row of blocks} of $A$ and all blocks with fixed pair of indices $(j,l)$ form a \emph{column of blocks} of $A$. A block $A_{ss'}^{tt'}$ is called a \emph{row-pivot} of $A$ if its rank is maximal among all ranks of blocks in its row of blocks, i.e.,
$$\rank\left(A_{ss'}^{tt'}\right)=\max_{j\in\supp(f),l=1,\ldots,f_j}\left\{\rank\left(A_{sj}^{tl}\right)\right\}%=\max_{i\in\supp(f),k=1,\ldots,f_i}\left\{\rank\left(A_{is'}^{kt'}\right)\right\}
.$$
A block $A_{ss'}^{tt'}$ is called a \emph{column-pivot} of $A$ if its rank is maximal among all ranks of blocks in its column of blocks, i.e.,
$$\rank\left(A_{ss'}^{tt'}\right)=\max_{i\in\supp(f),k=1,\ldots,f_i}\left\{\rank\left(A_{is'}^{kt'}\right)\right\}.$$
Since we perform elementary cyclic operation on columns the row-pivots are more important for us. So we will use the term \emph{pivot} instead of row-pivot in the rest of the paper. 

We fix pairs $(s,t)$ and $(s',t')$ such that $A_{ss'}^{tt'}$ is a pivot and write $r=\rank\left(A_{ss'}^{tt'}\right)$. Then the same arguments as those in the proof of \cite[Lem. 4]{Sha-1} show that $A$ can be transformed by elementary cyclic column operations into a matrix $\widehat{A}$ with blocks $\widehat{A}_{sj}^{tl}$ all $0$ for $(j,l)\neq (s',t')$ and $\widehat{A}_{ss'}^{tt'}$ has entries on the $r$-th diagonal (counting from the top-right corner) equal to $1$ and all the other entries equal to $0$. 

Let us consider an example illustrating the transformation $A\to \widehat{A}$.

\begin{exmp}
Continue with $B$ corresponding to $P=(4^2,3,2^2)$. Observe that for a generic matrix $A$ in $\Uc_B$ block $A_{44}^{21}$ is a pivot of rank equal to $4$ (i.e., generically $m_{30}\neq 0$). Using elementary cyclic column operations on the second row of blocks in $A$ we obtain that $\widehat{A}$ is of the form
    $$\left[\begin{array}{ccccccccccccccc}
     0&a_{11}&a_{12}&a_{13}&0&a_{21}&a_{22}&a_{23}&b_1&b_2&b_3&c_{11}&c_{12}&c_{21}&c_{22}\\
    0&0     &a_{11}&a_{12}&0&0     &a_{21}&a_{22}&0  &b_1&b_3&0  &c_{11}&0&c_{21}\\
    0&0&0&a_{11}&0&0&0&a_{21}&0&0&b_1&0&0&0&0\\
    0&0&0&0&0&0&0&0&0&0&0&0&0&0&0\\
    1     &0&0&0&0&0&0&0&0&0&0&0&0&0&0\\
    0     &1&0&0&0&0&0&0&0&0&0&0&0&0&0\\
    0     &   0  &1&0&0&0&0&0&0&0&0&0&0&0&0\\
    0     &0     &0     &1&0&0&0&0&0&0&0&0&0&0&0\\
     0     &f_{1} &f_{2} &f_{3} &0&g_{1}&g_{2}&g_{3}&0&h_1&h_2&i_{11}&i_{12}&i_{21}&i_{22}\\
    0     &0     &f_{1} &f_{2} &0&0    &g_{1}&g_2  &0&0  &h_1&0  &i_{11}&0&i_{21}\\
    0     &0     &0     &f_{1} &0&0    &0    &g_{1}&0&0&0&0&0&0&0\\
    0     &0     &j_{11}&j_{12}&0&0&k_{11}&k_{12}&0&l_{11}&l_{12}&0&m_{11}&0&m_{21}\\
    0     &0     &0     &j_{11}&0&0&0&k_{11}&0&0&l_{11}&0&0&0&0\\
    0     &0     &j_{21}&j_{22}&0&0&k_{21}&k_{22}&0&l_{21}&l_{22}&m_{30}&m_{31}&0&m_{41}\\
    0     &0     &0     &j_{21}&0&0&0&k_{21}&0&0&l_{21}&0&m_{30}&0&0
    \end{array}\right].$$
Here and later we use the letters only to indicate which entries are nonzero and which entries are equal. The exact values of an entry in $A$ and the corresponding entry in $\widehat{A}$ need not be equal.
\end{exmp}

\begin{lem}\label{lem on pivots}
    Suppose that $A$ is a generic matrix in $\Uc_B$, $P$ the Jordan type of $B$, $f=f(P)$ its frequency sequence and $z=\max\{i;\ i\in\supp(f)\}$. Then the pivots of $A$ are the following blocks:
    \begin{enumerate}
        \item[(a)] $A_{ii}^{k,k-1}$, $k=2,3,\ldots,f_i$, for any $i\in\supp(f)$ with $f_i\ge 2$,
        \item[(b)] $A_{z,z-1}^{1l}$, $l=1,2,\ldots,f_{z-1}$, if $f_{z-1}\neq 0$,
        \item[(c)] $A_{zz}^{1f_i}$,
        \item[(d)] $A_{ij}^{kl}$, $j\in\supp(f)$, $j>i$, and $l=1,2,\ldots,f_j$, for any $i\in\supp(f)$ and $k=1,2,\ldots,f_i$.
    \end{enumerate}
\end{lem}

\begin{proof}
    Item \textit{(a)} follows from the fact that $A_i^D$ is strictly lower-triangular and so $\rank\left(A_{ii}^{k,k-1}\right)=i$, maximal possible for the row of columns. Items \textit{(b)} and \textit{(c)} follow from the fact that $A_z^D$ is strictly lower-triangular and so we have $\rank\left(A_{z,z-1}^{1l}\right)=\rank\left(A_{zz}^{1f_z}\right)=z-1$, which is maximal for the first row of columns. For item \textit{(d)} observe that for any $j\in\supp(f)$ with $j>i$ and any $l$ we have $\rank\left(A_{ij}^{kl}\right)=i$, which is maximal possible in the row of blocks with $i\in\supp(f)$ and $k=1,2,\ldots,f_i$.
\end{proof}

\subsection{Jordan Type of the Restriction to a Generic Image}

The following theorem links the algebraic and combinatorial setups. Our main result Theorem 1 is then an immediate consequence. 

\begin{thm}\label{partial(P) = P(B_W)}
    Suppose that $B$ is a nilpotent matrix and that $P=P(B)$ is its Jordan type. Suppose that $A$ is a generic nilpotent matrix commuting with $B$ and that $W=\im A$ is its image. Then the Jordan type of the restriction $B|_W$ is given by $\df P$.
\end{thm}

\begin{proof}
    We may assume that $B$ is in Jordan canonical form and that $A\in\Uc_B$. Suppose that $f=f(P)$ is the frequency sequence of $P$ and that $z$ is the largest element of $\Ir(f)$. In other words, the largest blocks of $B$ are of size $z$. Suppose first that $f_z\ge 2$. By Lemma \ref{lem on pivots} blocks $A_{zz}^{k,k-1}$, $k=2,3,\ldots,f_z$, are pivots of $A$. Their rank is equal to $z$. After performing elementary cyclic column operations consecutively on the second, third, $\ldots, f_z$-th row of blocks of $A$, we obtain a matrix $\widehat{A}$ with only nonzero blocks $\widehat{A}_{zz}^{kl}$ for $k=2,3,\ldots,f_z$, those for $l=k-1$ and then $\widehat{A}_{zz}^{k,k-1}=I$. Assume now that $f_{z-1}=0$. Then by Lemma \ref{lem on pivots}(c) block $A_{zz}^{1f_z}$ is a pivot. After elementary cyclic column operations on the first row of blocks we get the only nonzero block $\widehat{A}_{zz}^{1f_z}$ in this row of blocks is equal to nilpotent Jordan block of size $z$. We obtain $f_z-1$ columns of blocks of rank $z$ and $1$ column of blocks (i.e., the $f_z$-th column of blocks) of rank $z-1$. We are done with the case $f_{z-1}=0$.
    
    Assume next that $f_{z-1}\neq 0$ and $f_z$ arbitrary. By Lemma \ref{lem on pivots} it follows then that $A_{z,z-1}^{1f_{z-1}}$ and $A_{z-1,z}^{f_{z}1}$ are pivots, and furthermore, if $f_{z-1}\ge 2$ then also $A_{z-1,z-1}^{k,k-1}$, $k=2,3,\ldots,f_{z-1}$, are pivots of $A$. After performing elementary cyclic column operations with these pivots, we obtain that all the blocks $\widehat{A}_{ij}^{kl}$ with $i=z, z-1$ and $j\le z-2$ are zero. Furthermore, each block of columns corresponding to these blocks has exactly one pivot. The rank of pivot blocks is equal to $z$ for the first $f_z-1$ columns of blocks and $z-1$ for the other $f_{z-1}+1$ columns of blocks. Observe that these pivot blocks are also column-pivots for their corresponding column of blocks. These follows from the assumption that blocks $A_i^D$ are strictly lower-triangular in $A$. 

    It remains to consider the case $f_z=1$ and $f_{z-1}=0$. Then there is only one pivot in the columns corresponding to Jordan chains of sizes $z$ and $z-1$ of $B$, the block $A_{zz}^{11}$. We perform elementary cyclic column operations on the first row of blocks, so that all its blocks except the first one are zero, and the first one is equal to the nilpotent Jordan block of size $z$. So, we do not have in this case any columns of blocks of rank $z$ and only $1$ column of blocks of rank $z-1$. Hence, we are done with the columns corresponding to Jordan chains of size $z$ and $z-1$ of $B$. 

    Next, we apply Lemma \ref{new_Jordan_basis} consecutively on each column of blocks and replace Jordan chains corresponding to blocks of sizes $z$ and $z-1$ with the chains given by $\widehat{A}$. These gives $f_z-1$ Jordan chains of length $z$ and $f_{z-1}+1$ Jordan chains of length $z-1$ for $B|_W$. These are exactly the frequencies of $z$ and $z-1$ in $\df(P)$.

    Let us pause the proof and consider again our example. 

\begin{exmp}\label{exmp:part3}
Suppose that $B$ is corresponding to $P=[4^2,3,2^2]$. Then $4$ is the maximal element in $\Ir(f)$ and we have $f_4=2$ and $f_3=1$. Observe that $A_{44}^{21}$, $A_{43}^{11}$ and $A_{34}^{21}$ are pivots of $A$. After performing elementary cyclic column operations with these pivots we get that $\widehat{A}$ is of the form
    $$\left[\begin{array}{ccccccccccccccc}
    0&0&0&0&0&0&0&0&1&0&0&0&0&0&0\\
    0&0&0&0&0&0&0&0&0&1&0&0&0&0&0\\
    0&0&0&0&0&0&0&0&0&0&1&0&0&0&0\\
    0&0&0&0&0&0&0&0&0&0&0&0&0&0&0\\
    1&0&0&0&0&0&0&0&0&0&0&0&0&0&0\\
    0&1&0&0&0&0&0&0&0&0&0&0&0&0&0\\
    0&0&1&0&0&0&0&0&0&0&0&0&0&0&0\\
    0&0&0&1&0&0&0&0&0&0&0&0&0&0&0\\
    0&0&0&0&0&1&0&0&0&0&0&0&0&0&0\\
    0&0&0&0&0&0&1&0&0&0&0&0&0&0&0\\
    0&0&0&0&0&0&0&1&0&0&0&0&0&0&0\\
    0     &0     &j_{11}&j_{12}&0&0&k_{11}&k_{12}&0&l_{11}&l_{12}&0&m_{11}&0&m_{21}\\
    0     &0     &0     &j_{11}&0&0&0&k_{11}&0&0&l_{11}&0&0&0&0\\
    0     &0     &j_{21}&j_{22}&0&0&k_{21}&k_{22}&0&l_{21}&l_{22}&m_{30}&m_{31}&0&m_{41}\\
    0     &0     &0     &j_{21}&0&0&0&k_{21}&0&0&l_{21}&0&m_{30}&0&0
    \end{array}\right].$$
A Jordan chain of $B$ corresponding to the block of size $4$ in $\widehat{A}$ is
$$B^{h-1}\left(v_{421}+\sum_{i-1}^2\sum_{j=1}^2 j_{ij}v_{2ij}\right), h=1,2,3,4,$$
and two Jordan chains of $B$ corresponding to blocks of size $3$ in $\widehat{A}$ are
$$B^{h-1}\left(v_{412}+\sum_{i-1}^2\sum_{j=1}^2 l_{ij}v_{2ij}\right), h=1,2,3,$$
and 
$$B^{h-1}\left(v_{311}+\sum_{i-1}^2\sum_{j=1}^2 k_{ij}v_{2ij}\right), h=1,2,3.$$
The frequencies of the restriction $B|_W$ are $f_4=1$ and $f_3=2$.
\end{exmp}

To conclude the proof of Theorem \ref{partial(P) = P(B_W)}, observe that the blocks $\widehat{A}_{ij}^{kl}$ are $0$ for $i=z,z-1$ and $j\in\supp(f)$ with $j\le z-2$. We restrict both $B$ and $\widehat{A}$ to the span $V'$ of all Jordan chains of $B$ of length at most $z-2$ and we repeat the above arguments by replacing the role of $B$ by $B|_{V'}$ and the role of $f(P)$ by $f'=(f_1,f_2,\ldots,f_{z-2}, 0,0)$. The process concludes when all the elements of $\Ir(f)$ are used. 
\end{proof}

\begin{exmp}
    In Example \ref{exmp:part3} we determined the frequencies of $4$ and $3$ of the restriction $B|_W$. It remains to consider the remaining columns of blocks. As done in the proof we restrict to the span $V'$ of chains of length $2$ (which are the only chains of lengths less than $3$). The corresponding $4\times 4$ matrix has now two pivots - blocks $A_{22}^{21}$ and $A_{22}^{12}$. After elementary cyclic column operations we obtain
    $$\left[\begin{array}{cccc}
        0 & 0 & 0 & 1 \\
        0 & 0 & 0 & 0 \\
        1 & 0 & 0 & 0 \\
        0 & 1 & 0 & 0 \\
    \end{array}\right].$$
    So, the restriction $B|_W$ has one chain of length $2$ and one chain of length $1$. All together, the frequency vector of $B|_W$ is equal to $(1,1,2,1)=\df(f(P))$. 
\end{exmp}

\begin{cor}[Theorem~\ref{thm:main}]
$\D(P) = \Des{P}$ for all $P \in \Pset$.
\end{cor}

\begin{proof}
Theorem~\ref{partial(P) = P(B_W)} gives $\D(\pt P) = \D(P)-1$ since the Jordan type of the restriction of a nilpotent matrix of Jordan type $T$ to its own image is equal to $T-1$.  Obviously $|\D(P)|=|P|$, so Lemma~\ref{lem:uniqueness} identifies $\D$ as the descent map.
\end{proof}

\subsection{Dominant Jordan type and Finite Fields}\label{finite fields}

 While our proof of Theorem~\ref{thm:main} does not require any assumptions about the nature of the underlying field $F$,
the definition of $\D(P)$  depends on the assumption that $F$ is  infinite.  Some care is needed in extending this definition to finite $F$. (Britnell and Wildon \cite[Prop. 4.12]{BW} showed that, in general, the pairs of conjugacy classes that occur for pairs of commuting matrices are field dependent.)

Theorem \ref{partial(P) = P(B_W)} relies on Burge's correspondence that is completely combinatorial and the results based on Shayman's ideas that are described in \S\ref{sec:nilpotentcomm}--\ref{sec:elementary}. A careful reading of the latter reveals that only solutions of linear matrix equations are used and elementary column operations are performed. Moreover, in the proof of Theorem \ref{partial(P) = P(B_W)}, we  use only the fact that all the free entries of an element in $\Uc_B$ are nonzero, and for this no assumption on the field is required. So,  over an arbitrary field (finite or infinite), we see that for any matrix $B$ of Jordan type $P$ there is an $A\in\Nil_B$ such that the Jordan type of $B|_{\im A}$ is equal to $\df P$.  For instance, matrix $A$ with only nonzero entries equal to $1$ on the first nonzero diagonal in pivots (as described in Lemma \ref{lem on pivots}) satisfies the condition. Together with Lemma~\ref{lem:uniqueness}, this implies that $\Des{P}$ occurs as the Jordan type of an element in $\Nil_B$ over any field $F$.

%Note that the Jordan type of a nilpotent matrix is preserved under a field extension $F\leq E$, so any orbit having a nonempty intersection with $\Nil_B$ over $F$ intersects nontrivially also the nilpotent commutator of $B$ over $E$.   However, it is not \emph{a priori} clear that the orbit corresponding to $\D(P)$ intersects $\Nil_B$, so we cannot

Over infinite fields it is known that  $\D(P)$ dominates any partition $R$ corresponding to an orbit having nonempty intersection with $\Nil_B$, i.e., $R\preceq \D(P)$ in the dominance order of partitions \cite[Prop. 3.2]{Kha-S}.  Suppose now that the field $F$ is finite, and take $E$ to be any infinite field extending $F$, e.g., the algebraic closure of $F$. Then, over $E$, Theorem~\ref{thm:main} identifies $\D(P)=\Des{P}$ as the unique dominant partition among Jordan types in $\Nil_B$. But over $F$, we know that $\Des{P}$ occurs as the Jordan type of an element in $\Nil_B$.  Hence $\Des{P}$ is also the unique dominant partition among Jordan types in $\Nil_B$ over $F$.  

In summary: Regardless of the ground field, the set of Jordan types in $\Nil_B$ always has a maximum element --- namely $\Des{P}$, where $P$ is the Jordan type of $B$.   The following therefore serves as an alternative definition of $\D(P)$ over any field (finite or infinite), under which Theorem~\ref{thm:main} is valid in all cases.

\begin{defn}
Suppose that $P\in\Pset$ is the Jordan type of a nilpotent matrix $B$. Then $\D(P)$ is the maximum  partition in the dominance order that occurs as a Jordan type of a matrix in $\Nil_B$. 
\end{defn}

\subsection{Additional Remarks}

Partition $\D(P)$ dominates all partitions that occur as a Jordan type of elements in $\Nil_B$. In particular, a partition $Q\in\RRset$ dominates all the partitions in $\D^{-1}(Q)$. Observe that the partition with the maximal number of parts in $\D^{-1}(Q)$ need not be dominated by all the partitions in $\D^{-1}(Q)$. For instance, the partition with $8$ parts in Example \ref{exmp-3} (see Figure \ref{fig:boxexample}) is not related in dominance order to any of partitions that are of the form  $[9,4^2,\ldots]$ in the figure. This shows that the dominance order on partitions does not relate as naturally to orbits in $\Nil_B$ as it does to orbits of all nilpotent matrices. Compare with the Gerstenhaber-Hesselink Theorem over complex numbers \cite[Thm. 6.2.5]{CoMcG}.

Observe also that if $Q=(q_1,q_2,\ldots,q_r) \in \RRset$ then the partition with the maximal number of parts in $\D^{-1}(Q)$ is equal to $[q_2+2,q_3+2,\ldots,q_r+2,1^{q_1-2r+2}]$. Thus, the answer to the question of Oblak in \cite[Rmk. 3, p. 612]{Obl-3} is affirmative.

%\john{I suggest we remove this paragraph and move straight into describing how everything works over an arbitrary field.  I can present my case when we meet.  If this bit stays, then I think the insistence that $Q$ and $R$ have the ``same size $n$'' is somewhat misleading, as the bijection holds whenever $Q$, $R$ have equal $\delta$-sets.} 

Finally, we remark that there is a natural symmetry mapping on partitions in $\D^{-1}(Q)$ obtained by swapping $i_j$ with $\delta_j-i_j+1$ for some or all $j$. Furthermore, if $Q$ and $R$ are two distinct partitions in $\RRset$ with equal set (not a sequence) of $\delta_i$; say, $\sigma$ is the permutation on $\delta_i$ of $Q$ so that $\delta_{\sigma(i)}$ form the sequence of $\delta_j$ for $R$. Then the bijection from $\D^{-1}(Q)$ to $\D^{-1}(R)$ sending partition in position $(i_j)_j$ to the partition in position $(i_{\sigma(j)})_j$ respects the number of parts. This is the bijection mentioned in \cite[Rmk. 1.2]{IKvSZ} from $\D^{-1}(u,u-r)$ to $\D^{-1}(u,r-1)$ for appropriate $u$ and $r$.

\section{The Oblak Process and Basili's Theorem}
\label{sec:oblak}

In this section we describe how Oblak's recursive computation of $\D(P)$ can be  understood through its interaction with the Burge operator $\pt$. In particular, we will see how Oblak's conjecture (i.e., correctness of the algorithm) follows from the universality of the descent map (Lemma~\ref{lem:uniqueness}).  Although the conjecture was recently settled by Basili \cite{Bas-I}, we believe  our   approach offers significant additional insight into the nature of the algorithm.

\subsection{Evaluation, Annihilation, and the Oblak Process}
\label{sec:process}

Oblak's algorithm was  formulated in  the context of Gansner-Saks theory, which connects the chain structure of a finite poset to the Jordan type of a generic nilpotent element in its incidence algebra.  By building on this combinatorial framework, Oblak conjectured that $\D(P)$ could be computed through a greedy recursive decomposition of a certain digraph associated with $P$.  We refer the  reader to~\cite{Kha-S} for a  survey of these developments.  

We shall instead recast Oblak's algorithm entirely in terms of partitions, or rather their frequency representations in $\Fset$.  The translation is straightforward and leads naturally to  the following definitions.

\begin{defn}
For $i \geq 1$ define  $\map{\rev{i}}{\Fset}{\N}$ and $\map{\ran{i}}{\Fset}{\Fset}$  by 
\[
	\rev{i}(f) := if_i+(i+1)f_{i+1} + 2 \sum_{j > i+1} f_j 
\]
and
\[
	\ran{i}(f) := (f_1,f_2,\ldots,f_{i-1},f_{i+2},f_{i+3},\ldots ).
\]
For technical reasons (to be seen later), we extend these definitions to $i=0$ as follows:
\begin{align*}
\rev{0}(f) &:= \rev{1}(f) = f_1 + 2(f_2+f_3+\cdots) \\
\ran{0}(f) &:= \ran{1}(f)=(f_3,f_4,\ldots).
\end{align*}
\end{defn}

We refer to $\rev{i}$ and $\ran{i}$ as   \emph{evaluation} and 
\emph{annihilation} at $i$, respectively.  The following relation is clear from the definitions.
\begin{lem}
\label{lem:valann}
For $f \in \Fset$ and $i \in \N$ we have\  $|\ran{i}(f)| = |f|-\rev{i}(f)$.
\end{lem}
For example,  $f=(3, 0, 2, 1, 3,0,1)$ gives $|f|=35$, 
$\rev{3}(f)=18$, $\ran{3}(f)=(3,0,3,0,1)$, and
$|\ran{3}(f)|=17$.

\begin{defn}
Let $f \in \Fset$.
We say index $i$ is \emph{maximal} for $f$ if $\rev{i}(f)$ is nonzero and is the maximum possible evaluation of $f$.  That is, $\rev{i}(f) = \max \{ \rev{j}(f) \,:\, j \in \N\} \neq 0$.
\end{defn}

Clearly $f$ admits a maximal index if and only if $f \neq \emptyf$.
The nonzero condition ensures that if $i$ is maximal for $f$ then $|\ran{i}(f)| < |f|$.

With these definitions in hand, Oblak's iterative algorithm can be described as follows:

\begin{defn}
\label{alg:oblak}
The  \emph{Oblak process} takes input $f \in \Fset$ and produces multiset output $[q_1,\ldots,q_k] \in \Pset$ according to the following procedure:
\begin{tabbing}
mm \= mmm\= mm\= mm\= mm\= mm\= mm\= \kill
\> 0 \> {\scshape{input}} $f \in \Fset$ \\
\> 1 \> \> $\obi{f}{0} \gets f$ \\
\> 2 \> \> $k \gets 1$ \\
\> 3 \> \> {\scshape{while} $\obi{f}{k-1} \neq \emptyf$} \\
\> 4 \> \> \>  choose a maximal index $i_k$ for $\obi{f}{k-1}$ \\
\> 5 \> \> \>  $q_{k} \gets \rev{i_k}(\obi{f}{k-1})$  \\
\> 6 \> \> \>  $\obi{f}{k} \gets \ran{i_k}(\obi{f}{k-1})$  \\
\> 7 \> \> \>  $k \gets k+1$  \\
\> 8 \> \> {\scshape{end while}} \\
\> 9 \> {\scshape{output}} $[q_1,\ldots,q_k]$
\end{tabbing}
\end{defn}

The table below shows the Oblak process  applied to input $f=(3,3,2,0,3,1,0,0,2)$, corresponding to partition  $P(f)=[1^3,2^3,3^2,5^3,6,9^2]$.  The output is $[q_1,q_2,q_3,q_4]=[25,17,10,2]$.   
$$
\begin{array}{l|l|l|c|c}
k & \obi{f}{k-1} & (\rev{0/1}, \rev{2}, \rev{3}, \ldots) & \text{ $i_{k}$} & q_{k} \\
\hline 
1 & (3,3,2,0,3,1,0,0,2) & (25,24,18,21,25,10,4,18,18,0,\ldots)
& 5 & 25 \\
2 & (3,3,2,0,0,0,2) & (17,16,10,4,4,14,14,0,\ldots)
& 1 & 17 \\
3 & (2,0,0,0,2) & (6,4,4,10,10,0,\ldots) & 4 & 10 \\
4 & (2) & (2,0,\ldots) & 1 & 2 \\
5 & \emptyf &  &  &  
\end{array}
$$
Alternatively, this same input could be processed by making different choices of $i_1$ and $i_3$:
$$
\begin{array}{l|l|l|c|c}
k & \obi{f}{k-1} & (\rev{0/1}, \rev{2}, \rev{3}, \ldots) & \text{ $i_{k}$} & q_{k} \\
\hline 
1 & (3,3,2,0,3,1,0,0,2) & (25,24,18,21,25,10,4,18,18,0,\ldots)
& 1 & 25 \\
2 & (2,0,3,1,0,0,2) & (14,15,17,8,4,14,14,0,\ldots)
& 3 & 17 \\
3 & (2,0,0,0,2) & (6,4,4,10,10,0,\ldots) & 5 & 10 \\
4 & (2) & (2,0,\ldots) & 1 & 2 \\
5 & \emptyf &  &  &  
\end{array}
$$
Note that these distinct computations yield the same result. It was shown by Khatami that this is always the case.
 
\begin{prop}[Khatami {\cite[Thm. 2.5]{Kha-1}}]
\label{prop:khatami}
The output of the Oblak process is independent of the choices of  maximal indices made at each step.
\end{prop}

This invariance permits the following definition.

\begin{defn}
For $f \in \Fset$, let $\oblak(f)$ denote the output of the Oblak process when applied to input $f$.  
%For $P \in \Pset$, let $\oblak(P):=\oblak(f(P))$.
\end{defn}
 
Let $q_1,\ldots,q_k$ and $\obi{f}{0},\ldots,\obi{f}{k}$ be as defined within the process.  Then Lemma~\ref{lem:valann} gives $|\obi{f}{j+1}|=|\obi{f}{j}|-q_{j+1}$ for all $j$ and thus $\sum_{i=1}^k q_i = |\obi{f}{0}|-|\obi{f}{k}|=|\obi{f}{0}|=|P|$.  It is also not difficult to see inductively that successive $q_i$ always differ by at least 2
\cite[Prop. 2.7]{Kha-1}.  Therefore $P \mapsto \oblak(f(P))$ is a size-preserving function from $\Pset$ to $\RRset$.  Oblak conjectured that this function is precisely $\D$, and the conjecture was settled more than a decade later by Basili.

\begin{thm}[Basili~\cite{Bas-I}]
\label{thm:basili}
$\D(P)=\oblak(f(P))$ for all $P \in \Pset$.
\end{thm}

We will now embark on an independent proof of  Theorem~\ref{thm:basili} that relies on the interaction between the Oblak process and  Burge encoding. Proposition~\ref{prop:khatami} will be deduced along the way.

\subsection{Oblak meets Burge}

As a motivating example, consider the following tables which show the Oblak process being applied to both $f = (3,0,1,1,0,0,0,1) \in \Fset$ and  its image $\pt f = (2,0,2,0,0,0,1)$ under the Burge operator $\pt$. 
\begin{align*}
&\begin{array}{lllcc}
k\ \ & \obi{f}{k-1} & (\rev{0/1}, \rev{2}, \rev{3}, \ldots) & \text{ $i_k$} & q_{k} \\
\hline 
1 & (3,0,1,1,0,0,0,1) & (9,7,9,6,2,2,8,8,0,\ldots)
& 0 & 9 \\
2 & (1,1,0,0,0,1) & (5,4,2,2,6,6,0,\ldots)
& 5 & 6 \\
3 & (1,1) & (3,2,0,\ldots) & 1 & 3 \\
4 & \emptyf &  &  &  \\
 &  &  &  &  \\
k\ \  & \obi{(\pt f)}{k-1} & (\rev{0/1}, \rev{2}, \rev{3}, \ldots) & \text{ $i_k$} & q_{k} \\
\hline 
1 & (2,0,2,0,0,0,1) & (8,8,8,2,2,7,7,0,\ldots)
& 0 & 8 \\
2 & (2,0,0,0,1) & (4,2,2,5,5,0,\ldots)
& 5 & 5 \\
3 & (2) & (2,0,\ldots) & 1 & 2 \\
4 & \emptyf &  &  & 
\end{array}
\end{align*}
Notice that the outputs $\oblak(f)=[9,6,3]$ and $\oblak(\pt f) = [8,5,2]$ satisfy 
\begin{equation*}
\label{eq:oblak_des}
\oblak(\pt f)=\oblak(f)-1.
\end{equation*}
Our aim is to prove that this relation holds true for every $f \in \Fset$. For then Lemma~\ref{lem:uniqueness} and Theorem~\ref{thm:main} would give $\oblak(f(P))=\Des{P}=\D(P)$, thus proving  Theorem~\ref{thm:basili}. 

The key lies in a more detailed analysis of the above example.   In particular, we note that the maximal indices $(i_1,i_2,i_3)=(0,5,1)$ chosen when processing $f$ remain valid maximal choices when processing $\pt f$. Moreover for each $k$ we have 
\begin{equation}
\label{eq:minus}
	\rev{i_k}(\obi{(\pt f)}{k-1})=\rev{i_k}(\obi{f}{k-1})-1
\end{equation}
and 
\begin{equation}
\label{eq:commute}
	 \obi{(\pt f)}{k-1} = \pt( \obi{f}{k-1} ).
\end{equation}
The situation is neatly summarized by the following commutative diagram, in which $\obi{f}{0},\ldots,\obi{f}{3}$ appear along the top row and $\obi{(\pt f)}{0},\ldots,\obi{(\pt f)}{3}$ along the bottom.
\[ \begin{tikzcd}
(3, 0, 1, 1, 0, 0, 0, 1) \arrow{r}{\ran{1}}\arrow[swap]{r}{\rev{1}=9} \arrow[swap]{d}{\pt}
 & (1, 1, 0, 0, 0, 1) \arrow{r}{\ran{5}} \arrow[swap]{r}{\rev{5}=6}\arrow[swap]{d}{\pt} 
 & (1, 1)\arrow{r}{\ran{1}} \arrow[swap]{r}{\rev{1}=3}\arrow[swap]{d}{\pt}  
 &  \ \emptyf  \arrow[swap]{d}{\pt}  \\
(2, 0, 2, 0, 0, 0, 1) \arrow{r}{\ran{1}}\arrow[swap]{r}{\rev{1}=8}
& (2, 0, 0, 0, 1) \arrow{r}{\ran{5}} \arrow[swap]{r}{\rev{5}=5}
& (2)  \arrow{r}{\ran{1}} \arrow[swap]{r}{\rev{1}=2} 
& \ \emptyf
\end{tikzcd}
\]
We could instead choose maximal indices $(i_1,i_2,i_3)=(3,5,1)$ when computing $\oblak(f)$.  These, too, remain maximal in the computation of $\oblak(\pt f)$, and both~\eqref{eq:minus} and~\eqref{eq:commute} continue to hold. This is captured by the diagram below. 
\[ \begin{tikzcd}
(3, 0, 1, 1, 0, 0, 0, 1) \arrow{r}{\ran{3}}\arrow[swap]{r}{\rev{3}=9} \arrow[swap]{d}{\pt} & (3, 0, 0, 0, 0, 1) \arrow{r}{\ran{5}} \arrow[swap]{r}{\rev{5}=6}\arrow[swap]{d}{\pt} & (3)\arrow{r}{\ran{1}} \arrow[swap]{r}{\rev{1}=3}\arrow[swap]{d}{\pt}  & \ \emptyf \arrow[swap]{d}{\pt}  \\
(2, 0, 2, 0, 0, 0, 1) \arrow{r}{\ran{3}}\arrow[swap]{r}{\rev{3}=8}& (2, 0, 0, 0, 1) \arrow{r}{\ran{5}} 
\arrow[swap]{r}{\rev{5}=5}& (2)  \arrow{r}{\ran{1}} \arrow[swap]{r}{\rev{1}=2}& \ \emptyf
\end{tikzcd}
\]
While these observations are suggestive of a general principal, some caution is required. The reader should confirm that $(i_1,i_2,i_3)=(1,6,1)$ is a maximal index sequence for $f$, but \emph{not} for $\pt f$.  

Based on the foregoing discussion we turn our attention to the following questions: 
\begin{itemize}
\item[(a)] If $i$ is maximal for $f \in \Fset$, under what conditions is $i$ maximal for $\pt f$?

\item[(b)]	For which $i$ do we have
$\rev{i}(\pt f) = \rev{i}(f)-1$ and $\ran{i}(\pt f) = \pt(\ran{i}(f))$? 
\end{itemize}
A thorough understanding of second question will allow us to address the first, so we begin with (b).  Our answer will be framed in terms of the following notions.

\begin{defn}
Let $f \in \Fset$.  We say $i \geq 1$ is \emph{right admissible} in $f$  if $f_i > 0$ and either $f_{i+1} > 0$ or $f_{i-1}=f_{i+1}=0$.
We say $i \geq 0$ is \emph{left admissible} in $f$ if $f_{i+1} > 0$ and either $f_{i} > 0$ or $f_{i}=f_{i+2}=0$.
\end{defn}

In other words,  $i$ is right admissible if $f_i \neq 0$ and $i$ is not the right end of a nontrivial spread, and $i$ is left admissible if $f_{i+1} \neq 0$ and $i+1$ is not the left end of a nontrivial spread. 
For example, the admissible indices for two elements of $\Fset$ are shown below. 
$$
\begin{array}{c|c|c}
f & \text{right} & \text{left} \\
\hline
(3,0,2,0,0,1,2,1,0,0,1) & 1, 3, 6, 7, 11 & 0, 2, 6, 7, 10 \\
(1, 2, 1, 0, 2, 1, 0, 1, 0, 2) & 1, 2, 5, 8, 10 & 1, 2, 5, 7, 9
\end{array}
$$

At first glance it seems that there is an asymmetry between right and left admissibility. However, this is merely a result of notational convention.  Both notions really concern the pair $(f_i, f_{i+1})$, which we choose to identify with index $i$. 

In fact  the only distinction in left vs. right admissibility lies in the treatment of trivial spreads. If $f$ has no trivial spreads, then the sets of left and right admissible indices coincide, whereas if $\{i\}$ is a trivial spread then $i$ is right admissible (but not left) and $i-1$ is left admissible (but not right).  In particular, note that $0$ is never right admissible and is left  admissible if and only if $f_1>0$ and $f_2=0$ (i.e., $\{1\}$ is a trivial spread).

\begin{prop}\label{prop-comdiagrams}
Let $f \in \Fset$, $i \geq 0$. If $i$ is either left admissible for $f$ or right admissible for $\pt f$, then  $\rev{i}(\pt f)=\rev{i}(f)-1$ and the following diagram commutes:
\begin{equation}
 \begin{tikzcd}
f \arrow{r}{\ran{i}} \arrow[swap]{d}{\pt} & g \arrow{d}{\pt} \\%
\pt f \arrow{r}{\ran{i}}& \pt g
\end{tikzcd}. \label{eq:comdiagram}\tag{$*$}
\end{equation}
\end{prop}

\begin{exmp}
Let $f=(3, 0, 2, 1, 0, 1, 2, 1)$, which has
$\pt f = (2, 0, 3, 0, 1, 0, 3)$. Then $\{0, 3, 6, 7\}$ is the set of left admissible indices in $f$, and $\{1, 3, 5, 7\}$ is
the set of right admissible indices in $\pt f$.  The proposition  guarantees that $\rev{i}(\pt f)=\rev{i}(f)-1$ and diagram~\eqref{eq:comdiagram}
commutes whenever $i\in \{0,1,3,5,6,7\}$. 

For example, at $i=1$ we have $\rev{1}(f)=17$, $\rev{1}(\pt f)=16$, $g=\ran{1}(f)=(2, 1, 0, 1, 2, 1)$ and
$\pt g = (3, 0, 1, 0, 3) = \ran{1}(\pt f)$.
It turns out that~\eqref{eq:comdiagram} also commutes when $i=4$, but it fails to commute when $i\in\{2,8\}$.  For instance, when $i=8$, we get
\begin{align*}
 \pt(\ran{8}(f))&=\pt(3, 0, 2, 1, 0, 1, 2) =(2, 0, 3, 0, 0, 2, 1),
\end{align*}
as compared to  $\ran{8}(\pt f) = (2, 0, 3, 0, 1, 0, 3)$.
\end{exmp}

The proof of Proposition~\ref{prop-comdiagrams} is  elementary but technical, involving the same essential argument  applied to various special cases.  The core idea is contained in the proof of the following Lemma.

\begin{lem}
\label{lem-comdiagrams}
Let $f \in \Fset$ and suppose $i \geq 1$ and $i+1 \in \Ir(f)$.  Then $\rev{i}(\pt f)=\rev{i}(f)-1$ and $\ran{i}(\pt f)=\pt(\ran{i}(f))$.
\end{lem}

\begin{proof}
Let $g = (f_{i+2},f_{i+3},\ldots)$.  Since $i+1 \in \Ir(f)$ we have $i+2 \not \in \Ir(f)$, or equivalently $g \in \AP$.  Therefore 
$$
	\pt f = ((\pt f)_1, \ldots, (\pt f)_{i-1},f_i+1,f_{i+1}-1, \pt g)
$$ 
and
\begin{align*}
\rev{i}(\pt f)
&=i(f_{i}+1)+(i+1)(f_{i+1}-1)+2\len{\pt g} \\
&=-1 + if_i +(i+1)f_{i+1} + 2\len{g} \\
&=\rev{i}(f)-1.
\end{align*}
Finally, $\ran{i}(\pt f)=
 ((\pt f)_1, \ldots, (\pt f)_{i-1}, \pt g) =
\pt(\ran{i}(f)).$
\end{proof}

\begin{proof}[Proof of Proposition~\ref{prop-comdiagrams}.] Let $i$ be right admissible in $\pt f$. 
If $i+1\in\Ir(f)$ then the conclusion follows from Lemma~\ref{lem-comdiagrams}, so we suppose that $i+1\not\in\Ir(f)$.  
\begin{itemize}
\item We first deal with the case  $i>1$.  Then we must have that both $i$ and $i+2$ belong to $\Ir(f)$.  Hence
\begin{eqnarray*}
\ran{i}(\pt f)&=& \ran{i}(\ldots, f_{i-1}+1, f_{i}-1, f_{i+1}+1, f_{i+2}-1,\ldots)\\
&=& (\ldots, f_{i-1}+1, f_{i+2}-1,\ldots) \\
&=& \pt(\ran{i}(f)),\\
\rev{i}(\pt f)&=& (f_i-1)i +(f_{i+1}+1)(i+1)+2\left((f_{i+2}-1)+\sum_{k>i+2} f_k\right)\\
&=& -1+ f_i i + f_{i+1}(i+1) +2\sum_{k>i+1} f_k \\
&=& \rev{i}(f)-1. 
\end{eqnarray*}
\item Now suppose that $i=1$.  We treat cases $(\pt f)_2>0$ and $(\pt f)_2=0$ separately.
\begin{itemize}
\item Assume $(\pt f)_2>0$. Since $2\not\in\Ir(f)$, this means that $3\in\Ir(f)$ and $(f_4,\ldots)\in\mathcal{A}$. Hence
\begin{eqnarray*}
\ran{i}(\pt f) &=& \ran{1}(f_1-1,f_2+1, f_3-1,\pt(f_4,\ldots))\\
&=& (f_3-1,\pt(f_4,\ldots))\\ 
&=& \pt(f_3, f_4,\ldots) = \pt(\ran{i} f),\\ 
\rev{i}(\pt f) &=& (f_1-1)(1)+(f_2+1)(2)+2\left((f_3-1)+\sum_{k>3} f_k\right)\\
&=& -1+ f_1 + 2f_2+2\sum_{k>2} f_k \\
&=& \rev{i}(\pt f)-1. 
\end{eqnarray*}
\item Assume now that $(\pt f)_2=0$.  Since $2\not\in\Ir(f)$ this implies that $f_2=0$.  Therefore we must have that $3\not\in\Ir(f)$ and $(f_3,\ldots)\in\mathcal{A}$, hence
\begin{eqnarray*}
\ran{i}(\pt f) &=& \ran{1}(f_1-1,0,\pt(f_3,\ldots))\\
&=& \pt(f_3,\ldots) = \pt(\ran{i} f),\\
\rev{i}(\pt f) &=& f_1-1+2\sum_{k>2} f_k = \rev{i}(f)-1.
\end{eqnarray*}
\end{itemize}
\end{itemize}

Now let $i$ be left admissible in $f$.
\begin{itemize}
\item We first deal with the case $i\ge 1$.  If $i+1\in\Ir(f)$, then the conclusion again comes from Lemma~\ref{lem-comdiagrams}, so we suppose $i+1\not\in\Ir(f)$. Since $i$ is left admissible in $f$, this means $f_{i}>0$, and both
$i,i+2\in\Ir(f)$. Hence
\begin{eqnarray*}
\ran{i}(\pt f)&=& \ran{i}(\ldots, f_{i-1}+1, f_{i}-1, f_{i+1}+1, f_{i+2}-1,\ldots)\\
&=& (\ldots, f_{i-1}+1, f_{i+2}-1,\ldots) = \pt(\ran{i}(f)),\mbox{ and }\\
\rev{i}(\pt f)&=& (f_{i}-1)i +(f_{i+1}+1)(i+1)+2\left((f_{i+2}-1)+\sum_{k>i+2} f_k\right)\\
&=& -1+ f_{i}i + f_{i+1}(i+1) +2\sum_{k>i+1} f_k = \rev{i}(f)-1. 
\end{eqnarray*}
\item We now analyse the case $i=0$.  In this case we must have that $f_1>0$ and $f_2=0$.  We analyse 
the cases $3\in\Ir(f)$ (i.e., $(f_3,\ldots)\in\mathcal{B}$) and $3\not\in\Ir(f)$ separately.
\begin{itemize}
\item Assume that $3\in\Ir(f)$.  Then
\begin{eqnarray*}
\lan{1}(\pt f) &=& (f_1-1,1,f_3-1,\pt(f_4,\ldots)) \\
&=& (f_3-1,\pt(f_4,\ldots)) = \pt(f_3,\ldots)= \pt(\lan{1} f),\\
\lev{1}(\pt f) &=& (f_1-1)+2+2\left((f_3-1)+\sum_{k>3} f_k\right)\\
&=& \lev{1}(f)-1.
\end{eqnarray*}
\item Assume that $3\not\in\Ir(f)$.  Then $(f_3,\ldots)\in\mathcal{A}$ and hence
\begin{eqnarray*}
\lan{1}(\pt f) &=& (f_1-1,0,\pt(f_3,\ldots)) \\
&=& \pt(f_3,\ldots) = \pt(\lan{1} f),\\
\lev{1}(\pt f) &=& (f_1-1)+2\sum_{k>2} f_k = \lev{1}(f)-1.
\end{eqnarray*}
\end{itemize}
\end{itemize}
\end{proof}

We now wish to extend Proposition~\ref{prop-comdiagrams} to address when  maximality of the index $i$ in $f$ ensures maximality of $i$ in $\pt f$.  This requires a somewhat roundabout approach.

\begin{defn}
\label{defn:equivalence}
Fix $f \in \Fset$.  We say indices $i,j  \in \N$ are \emph{equivalent}  with respect to $f$ if $\ran{i}(f)=\ran{j}(f)$ and $\rev{i}(f)=\rev{j}(f)$.
\end{defn}

For example, if $f=(3, 0, 2, 0, 0, 1, 1, 1, 0, 1)$,  then
the index set $\N$ is partitioned into equivalence classes $\{0,1\}$, $\{2,3\}$, $\{4\}$, $\{5\}$, $\{6,7\}$, $\{8,9,10\}$, and $\{11,12,\ldots\}$. We remark that the condition $\rev{i}(f)=\rev{j}(f)$ in Definition~\ref{defn:equivalence} is actually redundant as Lemma~\ref{lem:valann} shows it to follow from $\ran{i}(f)=\ran{j}(f)$.  
It is included only to emphasize the equality of both annihilations and evaluations at equivalent indices.

\begin{lem}
\label{lem:admissible_max}
Let $f \in \Fset$ with $f \neq \emptyf$.  If $i$ is maximal for $f$ then it is equivalent to some right admissible index $j$ and  also equivalent to some left admissible index $k$.
\end{lem}

\begin{proof}  Since $\rev{0}=\rev{1}$ we assume, with no loss of generality, that $i\ge 1$.

We first consider the case $f_i=0$.  Due to maximality of $\rev{i}(f)$ we cannot have $f_r=0$ for all $r>i$.  Let $j$ be the smallest index such that $j>i$ and $f_j>0$. From definition of $\rev{i}$ it is  clear that $\rev{i}(f)\le \rev{j}(f)$  with equality if and only if $j=i+1$ and $f_{i+2}=0$. Maximality of $\rev{i}(f)$ forces equality, in which case $i$ is left admissible and equivalent to the right admissible index $j=i+1$.

We now examine the case $f_i>0$.  If we also have $f_{i+1}>0$, then $i$ itself is both left and right admissible. Now suppose that $f_{i+1}=0$.
If $f_{i-1}>0$ then $\rev{i-1}(f)>\rev{i}(f)$, contradicting the maximality of $\rev{i}(f)$.  Hence we have $f_{i-1}=f_{i+1}=0$.  So $i$ is right admissible and equivalent to the left admissible index $k=i-1$.
\end{proof}

\begin{prop}\label{prop-maxeval}
Let $f \in \Fset$ with $f \neq (1)$. If index $i$ is left admissible and maximal for $f$, then $i$ is maximal for $\pt f$ and $\rev{i}(\pt f)=\rev{i}(f)-1$.
\end{prop}

\begin{proof}
Let $m_f=\rev{i}(f)$ be the maximum evaluation for $f$, 
and let $m_{\pt f}$ be the maximum evaluation value for $\pt f$. Left admissibility of $i$ implies $f \neq \emptyf$, and since $f \neq (1)$ we have  $\pt f \neq \emptyf$.  Lemma~\ref{lem:admissible_max} ensures that $\rev{j}(\pt f)=m_{\pt f}$ for some $j$ that is right admissible in $\pt f$.  Proposition~\ref{prop-comdiagrams} then gives both
\begin{align*}
1+m_{\pt f} &= 1+\rev{j}(\pt f) =\rev{j}(f)\le m_{f} \intertext{and}
1+m_{\pt f} &\ge 1+\rev{i}(\pt f) = \rev{i}(f) = m_{f},
\end{align*}
so  the inequalities to be tight and  $m_{\pt f}=\rev{i}(\pt f) = \rev{i}(f)-1$.
\end{proof}

\subsection{Oblak Chains and the proof of Basili's Theorem}

The ingredients are now in place to assemble the proof of Theorem~\ref{thm:basili}.  We begin by defining an \emph{Oblak chain} for $f \in \Fset$ to be any sequence $(\obi{f}{0},\ldots,\obi{f}{k})$ that could be generated through the Oblak process applied to $f$.  (Compare with  Definition~\ref{alg:oblak}.)

\begin{defn}
Let $f \in \Fset$.
An \emph{Oblak chain} for $f$ is a sequence $C=(\obi{f}{0}, \obi{f}{1}, \ldots, \obi{f}{k})$  such that $\obi{f}{0}=f$, 
$\obi{f}{k} = \epsilon$ and for every $r=1,\ldots, k$ we have $\obi{f}{r} = \ran{i_r}(\obi{f}{r-1})$ for some index $i_r$ that is maximal for $\obi{f}{r-1}$. (Note that $i_r$ is unique up to equivalence in $\obi{f}{r-1}$.) Such a sequence of indices $(i_1,\ldots, i_k)$ is called an \emph{index sequence} for $C$. The  multiset  
\begin{align*}
\rev{}(C)&:=[\rev{i_1}(\obi{f}{0}),\ldots, \rev{i_k}(\obi{f}{k-1})] 
=[\,|\obi{f}{0}|-|\obi{f}{1}|,\ldots, |\obi{f}{k-1}|-|\obi{f}{k}|\,],
\end{align*}
is called the \emph{valuation} of $C$.  
\end{defn}

Evidently, if $C=(\obi{f}{0},\ldots,\obi{f}{k})$ is an Oblak chain for $f$ then $\rev{}(C)$ is a partition of size $|f|$ and no two of the $\obi{f}{r}$ are equal.  Note that $C=(\emptyf)$ is the unique Oblak chain for $\emptyf$, with valuation $\rev{}(C)=\emptyf$.

We extend $\pt$ to act on Oblak chains as follows.

\begin{defn}  Let $C=(\obi{f}{0},\ldots,\obi{f}{k})$ be an Oblak chain.  We define
\[
\pt C: =\begin{cases} (\pt\obi{f}{0},\ldots, \pt\obi{f}{k}), & \obi{f}{k-1} \neq (1) \\
(\pt\obi{f}{0},\ldots, \pt\obi{f}{k-1}), &  \obi{f}{k-1}=(1)
\end{cases}.
\]
\end{defn}

\begin{thm}\label{thm:oblakburge} Let $f \in \Fset$ and let $C$ be an Oblak chain for $f$. 
Then $\pt C$ is an Oblak chain for $\pt f$ and 
$\rev{}(\pt C)=\rev{}(C)-1$.  
\end{thm}
\begin{proof} 
Let $C=(\obi{f}{0},\ldots,\obi{f}{k})$ and suppose $(i_1,\ldots,i_k)$ is an index sequence for $C$. By Lemma~\ref{lem:admissible_max}, each $i_r$ is equivalent in $\obi{f}{r-1}$ to a left admissible index $j_r$.  
Proposition \ref{prop-comdiagrams}  gives  
\begin{align*}
\pt\obi{f}{r}&=
	\pt(\ran{j_r}(\obi{f}{r-1}))=\ran{j_r}(\pt\obi{f}{r-1}), \\
\rev{j_r}(\pt\obi{f}{r-1})&=\rev{j_r}(\obi{f}{r-1})-1,
\end{align*}
and, if $\obi{f}{r-1} \neq (1)$ then  Proposition \ref{prop-maxeval} ensures $j_r$ is maximal for $\pt(\obi{f}{r-1})$.  

The condition $\obi{f}{r-1}=(1)$ implies $\obi{f}{r}=\emptyf$ and is therefore only possible at $r=k$. Moreover, it is clear that  $\rev{j_r}(\obi{f}{r-1})=1$ if and only if $\obi{f}{r-1}=(1)$. So  if indeed $\obi{f}{k-1} = (1)$, then 
$\pt C = (\pt\obi{f}{0},\ldots, \pt\obi{f}{k-1})$ is an Oblak chain for $\pt f$ with index sequence $(j_1,\ldots,j_{k-1})$ and valuation
\begin{align*}
	\rev{}(\pt C)&=[\rev{j_1}(\obi{f}{0})-1,
	\ldots,
	\rev{j_{k-1}}(\obi{f}{k-2})-1] \\
	&= [\rev{j_1}(\obi{f}{0}),
	\ldots,
	\rev{j_{k-1}}(\obi{f}{k-2}), 1]-1 \\
	&= \rev{}(C)-1.
\end{align*}
Otherwise, $\pt C = (\pt\obi{f}{0},\ldots, \pt\obi{f}{k})$ is an Oblak chain for $\pt f$ with index sequence $(j_1,\ldots,j_k)$ and valuation 
$\rev{}(\pt C)=[\rev{j_1}(\obi{f}{0})-1,
	\ldots,
	\rev{j_{k}}(\obi{f}{k-1})-1] = \rev{}(C)-1.$
\end{proof}

\begin{exmp}
\label{exmp:bigexample}
Let $P=[1,2^3,5,10,14]$ and $f=f(P)$.  Figure~\ref{fig:bigexample} illustrates the  interaction between the Oblak chains (displayed horizontally) and Burge chains (displayed vertically) obtained by repeated application of  Theorem~\ref{thm:oblakburge}, starting with a particular Oblak chain for $f$.  The labels above the horizontal arrows list the equivalent maximal indices at each step, with the left admissible ones in bold. The labels below the horizontal arrows are the corresponding evaluations, which are seen to decrease by one row-by-row down each column.

\begin{figure}
{\scriptsize
\[ \begin{tikzcd}
(1, 3, 0, 0, 1, 0, 0, 0, 0, 1, 0, 0, 0, 1)\adp\ara{\mathbf{13}, 14} \arb{14}  & (1, 3, 0, 0, 1, 0, 0, 0, 0, 1)\adp\ara{0, \mathbf{1}} \arb{11}  & (0, 0, 1, 0, 0, 0, 0, 1)\adp\ara{\mathbf{7}, 8} \arb{8}  & (0, 0, 1)\adp\ara{\mathbf{2}, 3} \arb{3}  & \adp\emptyf\\ 
(2, 2, 0, 1, 0, 0, 0, 0, 1, 0, 0, 0, 1) \adp\ara{\mathbf{12}, 13} \arb{13}  & (2, 2, 0, 1, 0, 0, 0, 0, 1) \adp\ara{0, \mathbf{1}} \arb{10}  & (0, 1, 0, 0, 0, 0, 1) \adp\ara{\mathbf{6}, 7} \arb{7}  & (0, 1) \adp\ara{0, \mathbf{1}, 2} \arb{2}  & \adp\emptyf\\ 
(3, 1, 1, 0, 0, 0, 0, 1, 0, 0, 0, 1) \adp\ara{\mathbf{11}, 12} \arb{12}  & (3, 1, 1, 0, 0, 0, 0, 1) \adp\ara{0, \mathbf{1}} \arb{9}  & (1, 0, 0, 0, 0, 1) \adp\ara{\mathbf{5}, 6} \arb{6}  & (1) \adp\ara{\mathbf{0}, 1} \arb{1}  & \emptyf \\ 
(2, 2, 0, 0, 0, 0, 1, 0, 0, 0, 1) \adp\ara{\mathbf{10}, 11} \arb{11}  & (2, 2, 0, 0, 0, 0, 1) \adp\ara{0, \mathbf{1}} \arb{8}  & (0, 0, 0, 0, 1) \adp\ara{\mathbf{4}, 5} \arb{5}  & \emptyf \adp\\ 
(3, 1, 0, 0, 0, 1, 0, 0, 0, 1) \adp\ara{\mathbf{9}, 10} \arb{10}  & (3, 1, 0, 0, 0, 1) \adp\ara{0, \mathbf{1}} \arb{7}  & (0, 0, 0, 1) \adp\ara{\mathbf{3}, 4} \arb{4}  & \adp \emptyf\\ 
(4, 0, 0, 0, 1, 0, 0, 0, 1) \adp\ara{\mathbf{8}, 9} \arb{9}  & (4, 0, 0, 0, 1) \adp\ara{\mathbf{0}, 1} \arb{6}  & (0, 0, 1) \adp\ara{\mathbf{2}, 3} \arb{3}  & \adp\emptyf \\ 
(3, 0, 0, 1, 0, 0, 0, 1) \adp\ara{\mathbf{7}, 8} \arb{8}  & (3, 0, 0, 1) \adp\ara{\mathbf{0}, 1} \arb{5}  & (0, 1) \adp\ara{0, \mathbf{1}, 2} \arb{2}  &  \adp\emptyf\\ 
(2, 0, 1, 0, 0, 0, 1) \adp\ara{\mathbf{6}, 7} \arb{7}  & (2, 0, 1) \adp\ara{\mathbf{0}, 1} \arb{4}  & (1) \adp\ara{\mathbf{0}, 1} \arb{1}  & \emptyf\\ 
(1, 1, 0, 0, 0, 1) \adp\ara{\mathbf{5}, 6} \arb{6}  & (1, 1) \adp\ara{0, \mathbf{1}} \arb{3}  & \emptyf \adp\\ 
(2, 0, 0, 0, 1) \adp\ara{\mathbf{4}, 5} \arb{5}  & (2) \adp\ara{\mathbf{0}, 1} \arb{2}  & \emptyf \adp\\ 
(1, 0, 0, 1) \adp\ara{\mathbf{3}, 4} \arb{4}  & (1) \adp\ara{\mathbf{0}, 1} \arb{1}  & \emptyf \\ 
(0, 0, 1) \adp\ara{\mathbf{2}, 3} \arb{3}  & \adp \emptyf\\ 
(0, 1) \adp\ara{0, \mathbf{1}, 2} \arb{2}  & \emptyf \adp\\ 
(1) \adp\ara{\mathbf{0}, 1} \arb{1}  & \emptyf\\ 
\emptyf
\end{tikzcd}
\]
}
\caption{An illustration of the iterated application of Theorem~\ref{thm:oblakburge}. (See Example~\ref{exmp:bigexample}.)}
\label{fig:bigexample}
\end{figure}
\end{exmp}

The following corollary of Theorem~\ref{thm:oblakburge} says that the output of the Oblak process is  not dependent on the choices of maximal indices made along the way. That is, we have recovered Proposition~\ref{prop:khatami}.

\begin{cor}[Proposition~\ref{prop:khatami}]
\label{cor:khatami2}
Let $f \in \Fset$.  Any two Oblak chains for $f$ have the same valuation.
\end{cor}
\begin{proof}  We proceed by induction on the size of $f$.  The statement is obvious if $f=\emptyf$.  Suppose $f \neq \emptyf$ and let $C_1, C_2$ be Oblak chains for $f$.  Then Theorem~\ref{thm:oblakburge} implies $\pt C_1$ and $\pt C_2$ are Oblak chains for $\pt f$, with $\rev{}(\pt C_i)=\rev{}(C_i)-1$.  The inductive hypothesis gives $\rev{}(\pt C_1)=\rev{}(\pt C_2)$, hence $\rev{}(C_1)-1=\rev{}(C_2)-1$.  But  this implies $\rev{}(C_1)=\rev{}(C_2)$, since $\rev{}(C_1)$ and $\rev{}(C_2)$ are of common size $|f|$.   (In particular, if $\rev{}(C_i)-1=[1^{m_1},2^{m_2},\ldots]$ then $\rev{}(C_i)=[1^k, 2^{m_1}, 3^{m_2},\ldots]$, where $k=|f|-\sum_{i \geq 2} (i+1)m_i$.) 
\end{proof}

With Corollary~\ref{cor:khatami2} in hand we can unambiguously regard $\rev{}$ as a function from $\Fset$ to $\Pset$ by setting $\rev{}(f):=\rev{}(C)$, where $C$ is any Oblak chain for $f$.  Of course this identifies $\rev{}(f)$ with $\oblak(f)$, and we arrive at Basili's Theorem  as promised.

\begin{cor}[Theorem~\ref{thm:basili}]
$\D(P)=\rev{}(f(P))$ for any $P \in \Pset$.
\end{cor}

\begin{proof}
Let $\rev{}(P):=\rev{}(f(P))$. We know $|\rev{}(P)|=|P|$ and according to Theorem~\ref{thm:oblakburge} we have $\rev{}(\pt P) = \rev{}(P)-1$.  Hence Lemma~\ref{lem:uniqueness} gives $\rev{}(P)=\Des{P}$, and the result follows by Theorem~\ref{thm:main}.
\end{proof}

\section*{Acknowledgments}
T.~Ko\v{s}ir was supported in part by the Research and Innovation Agency of Slovenia (research core funding P1-0222, research projects N1-0154 and J1-3004 until December 31, 2023, and research core funding P1-0448 and research project N1-0217 after January 1, 2024).
M.~Mastnak was supported in part by the NSERC Discovery Grant 371994-2019.

Most of research presented in the paper was done during an extended visit of T. Ko\v sir at Saint Mary's University in Halifax. He is truly grateful for the hospitality he received during his stay.

\bibliographystyle{plain}

\end{document}